\documentclass[a4paper,10pt]{article}
\usepackage{authblk}

\usepackage[ngerman,english]{babel}
\usepackage[latin1]{inputenc}
\usepackage{csquotes}
\usepackage{amsfonts,amsmath,amsthm}
\usepackage{empheq}
\usepackage{cancel}
\usepackage[titletoc,title]{appendix}
\usepackage[backend=bibtex8,doi=false,eprint=false,firstinits=true,isbn=false,style=numeric-comp,maxnames=99]{biblatex}
\makeatletter
\def\blx@maxline{77}
\makeatother

\bibliography{Bib_ECM_Degradation.bib}

\usepackage{cases}
\usepackage{mathabx}
\usepackage{bbm}
\usepackage{xfrac}
\usepackage{fancyhdr}
\usepackage{color}
\usepackage[colorinlistoftodos,textsize=small]{todonotes}
\usepackage[colorlinks]{hyperref}
\definecolor{blue75}{rgb}{0,0,.75}
\definecolor{green75}{rgb}{0,.75,0}
\hypersetup{colorlinks=true, urlcolor=blue75,linkcolor=blue75,citecolor=green75,pdfstartview=FitB,bookmarksopen=true,bookmarksopenlevel=1}
\usepackage[a4paper, left=2.5cm, right=2.0cm, top=2.5cm,bottom=2cm]{geometry}
\usepackage{constants}
\newcommand{\parenthezises}[1]{\arabic{#1}}
\newconstantfamily{C}{
	symbol=C,
	format=\parenthezises,
	reset={section}
}
\newconstantfamily{M}{
	symbol=M,
	format=\parenthezises,
	reset={section}
}
\newconstantfamily{B}{
	symbol=B,
	format=\parenthezises,
	reset={section}
}
\newconstantfamily{xi}{
	symbol=\xi,
	format=\parenthezises,
	reset={section}
}
\usepackage{enumerate}
\usepackage{graphicx}
\graphicspath{ {images/} }
\usepackage{wrapfig}
\usepackage{figbib}
\allowdisplaybreaks
\usepackage[capitalise]{cleveref}

\crefdefaultlabelformat{{\it #2#1#3}}

\crefname{equation}{}{}

\crefname{enumi}{}{}
\creflabelformat{enumi}{{(#2#1#3)}}

\crefname{section}{{\it Section}}{{\it Sections}}
\crefname{subsection}{{\it Subsection}}{{\it Subsections}}
\crefname{subsubsection}{{\it Paragraph}}{{\it Paragraphs}}
\usepackage[pagewise]{lineno} 

\newtheorem{Theorem}{Theorem}[section]
\crefname{Theorem}{{\it Theorem}}{{\it Theorems}}
\newtheorem{Definition}[Theorem]{Definition}
\crefname{Definition}{{\it Definition}}{{\it Definitions}}
\newtheorem{Lemma}[Theorem]{Lemma}
\crefname{Lemma}{{\it Lemma}}{{\it Lemmas}}

\crefname{Proposition}{{\it Proposition}}{{\it Propositions}}

\crefname{Assumption}{{\it Assumption}}{{\it Assumptions}}

\crefname{Corollary}{{\it Corollary}}{{\it Corollaries}}

\theoremstyle{definition}
\newtheorem{Remark}[Theorem]{Remark}
\crefname{Remark}{{\it Remark}}{{\it Remarks}}

\crefname{Notation}{{\it Notation}}{{\it Notations}}

\crefname{Example}{{\it Example}}{{\it Examples}}

\newcommand{\be}{\begin{equation} \label}
\newcommand{\ee}{\end{equation}}
\newcommand{\bea}{\begin{eqnarray}\label}
\newcommand{\eea}{\end{eqnarray}}
\newcommand{\bas}{\begin{eqnarray*}}
	\newcommand{\eas}{\end{eqnarray*}}
\newcommand{\bit}{\begin{itemize}}
	\newcommand{\eit}{\end{itemize}}
\begin{document}
	\enlargethispage{10mm}

\title{Global existence of weak solutions to a tissue regeneration model}
\author{Nishith Mohan\footnote{mohan@mathematik.uni-kl.de}\quad  and\quad  Christina Surulescu\footnote{surulescu@mathematik.uni-kl.de}\\
	{\small RPTU Kaiserslautern-Landau, Department of Mathematics,} \\
	{\small Gottlieb-Daimler-Str. 48, 67663 Kaiserslautern, Germany}}

\date{\today}
\maketitle

\begin{abstract} We study a cross-diffusion model for tissue regeneration which involves the dynamics of human mesenchymal stem cells interacting with chondrocytes in a medium containing  a differentiation factor. The latter acts as a chemoattractant for the chondrocytes and influences the (de)differentiation of both cell phenotypes. The stem cells perform haptotaxis towards extracellular matrix expressed by the chondrocytes and degraded by themselves. Cartilage production as part of the extracellular matrix is ensured by condrocytes. The growth factor is provided periodically, to maintain the cell dynamics. We provide a proof for the global existence of weak solutions to this model, which is a simplified version of a more complex setting deduced in \cite{surulescu_AMM}.  
\end{abstract}	

\section{Introduction}

Tissue regeneration has attracted increasing interest during the last decades, due to poor mid- and long-term outcomes of treatments focusing on resection \cite{Paxton2011,Stein2010}. This led, e.g., in the context of meniscal tears to a paradigm shift in therapeutic approaches which currently promote healing by repair or regeneration \cite{Makris2014,Pillai2018}. The quest for appropriate implants serving as support for cartilage regeneration is ongoing; very few products are (commercially) available and each of them has its drawbacks, see \cite{surulescu_AMM} for a very concise review. Several aspects have to be taken into account on the way towards an optimal artificial scaffold. Thus, not only the physical and structural properties of the material are relevant, but also how the main cell types involved in the production and degradation of tissue components interact with the scaffold and its embedding environment and correspondingly adapt their migration, proliferation, (de)differentiation, and expression of extracellular matrix (ECM). \\[-2ex]

\noindent
Relatively few mathematical models accounting for (some of) these aspects are available; we refer to \cite{ODea2012,WSE-H} for reviews of modeling in tissue regeneration and engineering in a larger framework and to \cite{Burova2019} for models dedicated to bone tissue engineering. The vast majority of the continuous models are either multiphase approaches, where the cell populations and the tissue are components of a mixture also containing fluid(s) in which chemical cues are dissolved, see e.g.  \cite{klika2016overview,Borgiani2017} and references therein, or involve reaction-diffusion(-transport) equations (RD(T)Es). The former category has the advantage of being able to include biomechanical effects in a more detailed way, but the rigorous mathematical analysis of such settings is challenging and rarely addressed. \cite{jackson2002mechanical,Kumar2022} established connections (in 1D and higher dimensions, respectively) between multiphase and RDTE models in a biologically different, but mathematically closely related framework. 
 The settings presented in \cite{Campbell2019a,Campbell2019b} and following the one introduced in \cite{Lutianov2011} focus on the dynamics of pre-cultured mesenchymal stem cells (MSCs) seeded into the defect region, coupled with the evolution of (co-implanted) chondrocytes and growth factors/nutrients, and newly formed ECM. Those models are one-dimensional in space, require diffusion of ECM, and concentrate on performing numerical simulations to assess the effect of various cell implantation scenarios and composition of the extracellular space. For models in higher dimensions we refer, e.g., to \cite{andreykiv2008simulation,BAILONPLAZA2001,geris2008angiogenesis}, of which the latter two also feature haptotaxis of MSCs towards gradients of ECM, \cite{geris2008angiogenesis} also including chemotaxis of MSCs, fibroblasts, and endothelial cells towards gradients of growth factors. All these models have been set up in a heuristic manner, directly on the macroscale where space-time dynamics of volume fractions for cells, tissues, and chemoattractants are studied. Recently a multiscale approach was employed in \cite{surulescu_AMM} to deduce a complex macroscopic model for MSC and chondrocyte dynamics in an artificial PET scaffolds contained in a bioreactor. The deduction mehod follows previous works in the context of cell migration through anisotropic tissue  \cite{engwer2015glioma,engwer2016effective,conte2023mathematical,corbin2021modeling,Hillen2006}; it starts from microscale dynamics on the subcellular level and uses the mesoscopic description via kinetic transport equations for the MSC and chondrocyte density functions to obtain by parabolic upscaling reaction-diffusion-taxis equations on the population level of space-time dependence. Therein, the MSC motility terms involve a cell diffusion tensor which carries information about the scaffold's fibre distribution. The development is informal, but we refer to \cite{zhigun2022novel} for a rigorous result in a much  simpler setting. Beside dynamics of MSC and chondrocytes interacting with fibre bound proteins, newly produced ECM, and a differentiation medium, the model in \cite{surulescu_AMM} also involves fluid flow and therewith associated deformations of the scaffold. The anisotropic structure of the scaffold is accounted for by way of statistical estimation of the directional fibre distribution performed on CT data. The precursor models in \cite{Grosjean23,JaegerGrosjeanPlunderetal.2024} are simpler descriptions of the same biological problem: meniscus cartilage regeneration. All mentioned works focus on numerical simulations and investigations of various aspects of the considered dynamics. To our knowledge, \cite{Mohanan2025,Mohan2026} are the first works to address in this context analytical issues of reaction-diffusion-taxis models for cell migration, (de)differentiation, and spread in a heterogeneous environment. In those simplified settings an effective description of the scaffold is replaced by the dynamics of hyaluron impregnating the scaffold fibres. Thus, \cite{Mohanan2025} studied global existence of classical solutions for a model with MSCs performing taxis towards fibre-bound hyaluron gradients, as well as pattern formation, showing that patterns were driven by the mentioned taxis. The very recent result in \cite{Mohan2026} shows global existence of weak solutions to a more complex system with double haptotaxis of MSCs towards gradients of hyaluron and of newly-formed tissue produced by chondrocytes.\\[-2ex]
 
 \noindent
 In this note we consider yet another version of the model part in \cite{surulescu_AMM} which describes dynamics of cells, ECM, and differentiation medium and further simplify it - however also allowing for tactic behavior of chondrocytes.  \\[-2ex]
 
 \noindent
The rest of the paper is organized as follows: in Section \ref{sec:model-maintho} we set up and explain the model and present the main result, which claims the global existence of weak solutions to the introduced model. Section \ref{sec:approx-probl} introduces a sequence of regularized problems, which is supposed to approximate the actual one. In Section \ref{sec:GE_approx} the global existence of solutions to the approximate problems is shown. Section \ref{sec:entropy} is dedicated to obtaining estimates which stem from an entropy-type functional and are essential for the passage to limits and therewith  associated construction of weak solutions in Section \ref{sec:limits}.

\section{Model set-up and statement of the main result}\label{sec:model-maintho}

%

We consider the following model:
  \begin{equation}
  \begin{cases}
  \label{model}
    \partial_t c_1 = a_1 \Delta c_{1} - \nabla \cdot (b_\tau c_1 \nabla \tau) - \alpha_1(\chi) \frac{c_1}{1 + c_1} + \alpha_2(\chi) \frac{c_2}{1 + c_2}  + \beta c_{1}(1 - c_{1} - c_{2} - \tau),                             &  x \in \Omega,\ t > 0,           \\[7pt]
    \partial_t c_2  = a_2 \Delta c_2 - \nabla \cdot (b_\chi c_2 \nabla \chi) + \alpha_1(\chi) \frac{c_1}{1 + c_1} -\alpha_2(\chi) \frac{c_2}{1 + c_2} ,                                                                       &  x \in \Omega,\ t > 0,           \\[7pt]
    \partial_t \chi = D_\chi \Delta \chi - a_\chi (c_1 + c_2) \chi + F(\chi),                                                                                                                      &  x \in \Omega,\ t > 0,           \\[7pt]
    \partial_t \tau = - \delta c_1 \tau - \mu \tau +  \tfrac{c_2}{1 + c_2},                                                                                                                  &  x \in \Omega,\ t > 0,
  \end{cases}
\end{equation}
subject to zero-flux boundary conditions ($\nu$ denotes the outward unit normal on the boundary of $\Omega$ )
\begin{equation}
  \label{boundary_conditions}
  a_1 \frac{\partial c_1}{\partial \nu} - b_\tau c_1 \frac{\partial \tau}{\partial \nu} = \frac{\partial c_2}{\partial \nu} = \frac{\partial \chi}{\partial \nu} = 0,  \quad x \in \partial \Omega, ~~ t > 0,
\end{equation}
and initial conditions
\begin{equation}
  \label{initial_conditions}
  c_1(x, 0) = c_{10}(x), \quad c_2(x, 0) = c_{20}(x), \quad \chi(x, 0) = \chi_0(x) \quad \tau(x, 0) = \tau_0(x) \quad x \in \Omega.
\end{equation}
Thereby, $c_1$ and $c_2$ represent volume fractions of MSC and chondrocyte cell populations, respectively, $\tau$ is the density of ECM expressed by chondrocytes, and $\chi $ denotes the concentration of a differentiation medium. The latter induces and sustains differentiation of MSCs to chondrocytes and the phenotype preservation of the latter. It diffuses throughout the whole region containing the cells and ECM, is uptaken by both cell types, and has to be periodically supplied from outside, in order to prevent dedifferentiation of chondrocytes and to maintain the cell populations and their dynamics.  The supply $F(\chi )$ of differentiation medium  will be addressed in more detail in \eqref{F_assump} below. In our setting, MSCs are co-seeded and co-cultured with chondrocytes (typically with MSCs clearly dominating their differentiated counterparts).\\[-2ex]

\noindent
Both types of cells diffuse and can infer phenotypic switch by (de)differentiation, with the corresponding rates depending on the amount of available $\chi$ and with intrinsic limitations. MSCs perform haptotaxis towards ECM; they try to adapt their direction of motion to local cues in the tissue. They also proliferate; here we assume this to happen in a logistic manner, with intra- and interspecific restrictions. \\[-2ex]

\noindent
Chondrocytes produce cartilage (with volume fraction $\tau$), which is a specialized part of the ECM. This production is inferring saturation when too high levels on $c_2$ become available. In fact, the differentiation medium also contains growth factors, which act as chemoattractants for the chondrocytes. In order not to complicate the setting we lump such soluble components in the 'differentiation medium' notion. This motivates the second term on the right hand side of the $c_2$ equation in \eqref{model}. Unlike MSCs, mature chondrocytes have limited proliferation capacity. In vivo, they are relatively quiescent and proliferate slowly, if at all, especially in healthy adult cartilage \cite{Loeser2009}, while in vitro they can infer division to some extent, but tend to quickly dedifferentiate, thus losing their ability to produce cartilage, see e.g.  \cite{VONDERMARK1977}. Therefore, we do not include a proliferation term in the second equation of \eqref{model}.\\[-2ex]

\noindent
Eventually, ECM is expressed by chondrocytes - as mentioned, is degraded in a natural manner - with rate $\mu$, and degraded by MSCs, e.g. by expression of matrix degrading enzymes.\\[-2ex]

\noindent
Model \eqref{model} originates from that obtained in \cite{surulescu_AMM}, but it is different, in the sense that we replaced here the myopic MSC and chondrocyte diffusions by linear ones, took constant motility coefficients instead of those involving the cell diffusion tensors, and let the ECM performing only haptotaxis towards gradients of newly produced ECM. Thus we do not account here for (indirect) scaffold dynamics via evolution of fibre bound proteins and taxis of MSCs towards such gradients. Instead, $c_2$ cells perform chemotaxis towards differentiation medium (more precisely toward chemical cues contained therein). \\[-1ex]

\noindent
Further, we assume that
\begin{equation}
\label{main_ic_assumptions}
  \begin{cases}
    c_{10}, c_{20} \in C^0(\bar{\Omega}), \quad \chi_0, \tau_0 \in W^{1, 2}(\Omega) \cap C^0(\bar{\Omega}), \\[2mm]
    c_{10}, c_{20} \geq 0, \quad \chi_0, \tau_0 > 0 \text{ in } \Omega, \quad c_{10} \not\equiv 0, \quad c_{20} \not\equiv 0,
  \end{cases}
\end{equation}
the functions $\alpha_i$, $i \in \{1, 2\}$, satisfy
\begin{equation}
\label{alpha_def}
\begin{cases}
  \alpha_i(z) \in C^{\vartheta, \frac{\vartheta}{2}}(\bar{\Omega} \times [0, T]), \quad (\vartheta \in (0, 1), ~ T > 0), \quad \alpha_i(z) > 0, \\[1mm]
  \alpha_i(z) \leq M_{\alpha_i}, \quad i \in \{1, 2\}, \quad \text{for all } z \geq 0,
\end{cases}
\end{equation}
and $F(\chi)$ is defined as
\begin{align}
\label{F_assump}
F(\chi)(x, t) := \chi(x, t \in \mathcal{T}_\chi) = \chi_0 \frac{1}{|\Omega|} \mathbbm{1}_\Omega(x),
\end{align}
with $\mathcal{T}_\chi$ a finite set of predefined time points (in days) \footnote{In the experiments performed by our A. Ott and
	G. Schmidt at the Deutsche Institute f\"ur Textil- und Faserforschung (DITF) in Denkendorf the differentiation medium was provided every 3rd day, over a total time span of 3 weeks.}. It models the fact that the differentiation medium is provided at several different times
during the experiment, each time the same overall quantity $\chi _0$, which is supposed to quickly diffuse within
the whole domain $\Omega $. We therefore consider it to be uniformly distributed. \\[-2ex]

\noindent
Moreover, all parameters $a_1, a_2, a_\chi, b_\tau, b_\chi, D_\chi, \beta, \delta,$ and $\mu$ are positive.\\[-1ex]

\noindent
The primary objective of this work is to construct global weak solutions for the problem \eqref{model}, \eqref{boundary_conditions}, and \eqref{initial_conditions}. To this end, we will first define weak solutions to problem \eqref{model}-\eqref{initial_conditions}.
\begin{Definition}
  \label{solution_definition}
  Let $T \in (0, \infty)$. A weak solution to the problem \eqref{model}-\eqref{initial_conditions} in $\Omega \times (0, T)$ consists of a quadruple of nonnegative functions $(c_1, c_2, \chi, \tau)$ such that
  \begin{equation*}
    \begin{cases}
      c_1 \in L^2(\Omega \times (0, T)) \cap L^{\frac{4}{3}}(0, T; W^{1, \frac{4}{3}}(\Omega)),             \\[5pt]
      c_2 \in L^{\frac{5}{4}}(0, T; W^{1, \frac{5}{4}}(\Omega)),             \\[5pt]
      \chi \in L^2(0, T; W^{1, 2}(\Omega)),   \quad ~\text{and}     \\[5pt]
      \tau \in L^\infty(\Omega \times (0, T)) \cap L^2(0, T; W^{1, 2}(\Omega)),
    \end{cases}
  \end{equation*}
  and satisfy the equations
\begin{align}
  \label{defeq1}
  \nonumber  - \int_0^T \int_\Omega c_{1} \partial_t \psi - \int_\Omega c_{10} \psi(\cdot, 0) & = - a_1 \int_0^T \int_\Omega \nabla c_{1} \cdot \nabla \psi + b_\tau \int_0^T \int_\Omega c_{1} \nabla \tau \cdot \nabla \psi  - \int_0^T \int_\Omega \alpha_1(\chi)\frac{c_{1}}{1 + c_{1}} \psi\\[5pt]
  & + \int_0^T \int_\Omega \alpha_2(\chi)\frac{c_{2}}{1 + c_2} \psi + \beta \int_0^T \int_\Omega c_{1}(1 - c_{1}  -  c_{2}  - \tau )\psi
\end{align}
and
\begin{align}
  \label{defeq2}
 \nonumber  & - \int_0^T \int_\Omega c_{2} \partial_t \psi - \int_\Omega c_{20} \psi(\cdot, 0) = - a_2 \int_0^T \int_\Omega \nabla c_{2} \cdot \nabla \psi  - b_\chi \int_0^T \int_\Omega c_{2} \chi \Delta \psi \\[5pt]
  & - b_\chi \int_0^T \int_\Omega \chi \nabla c_{2 } \cdot \nabla \psi + \int_0^T \int_\Omega \alpha_1(\chi)\frac{c_{1}}{1 + c_{1}} \psi - \int_0^T \int_\Omega \alpha_2(\chi)\frac{c_{2}}{1 + c_{2}} \psi
\end{align}
and
\begin{align}
  \label{defeq3}
 - \int_0^T \int_\Omega \chi \partial_t \psi - \int_\Omega \chi_{0} \psi(\cdot, 0) = - D_\chi \int_0^T \int_\Omega \nabla \chi \cdot \nabla \psi - a_\chi \int_0^T \int_\Omega c_{1}\chi \psi - a_\chi \int_0^T \int_\Omega c_{2}\chi \psi + \int_0^T \int_\Omega F(\chi)
\end{align}
as well as
\begin{align}
  \label{defeq4}
   - \int_0^T \int_\Omega \tau \partial_t \psi - \int_\Omega \tau_{0} \psi(\cdot, 0) = - \delta \int_0^T \int_\Omega \tau c_{1} \psi - \mu \int_0^T \int_\Omega \tau \psi - \int_0^T \int_\Omega \dfrac{c_{2}}{1 + c_{2}} \psi
\end{align}
for all $\psi \in C^\infty_0(\bar{\Omega} \times [0, T))$ with $\frac{\partial \psi}{\partial \nu} = 0$ on $\partial \Omega \times [0, T)$. If the quadruple $(c_1, c_2, \chi, \tau)$ is a weak solution to \eqref{model}-\eqref{initial_conditions} in $\Omega \times (0, T)$ for all $T > 0$, then it is referred to as a global weak solution.
\end{Definition}

\noindent
Our main result asserts that problem \eqref{model}-\eqref{initial_conditions} admits a global weak solution:

\begin{Theorem}
  \label{main_theorem}
Let $n \leq 3$ and let $\Omega \subset \mathbb{R}^n$ be a bounded domain with smooth boundary. Assume that \eqref{main_ic_assumptions} holds, that $\alpha_i$ for $i \in \{1, 2\}$ satisfy \eqref{alpha_def}, and that $F$ satisfies \eqref{F_assump}. Then, problem \eqref{model}--\eqref{initial_conditions} admits at least one global weak solution in the sense of Definition \ref{solution_definition}.
\end{Theorem}
In the sequel me make the following notations and conventions:
\begin{itemize}
  \item The integrals $\int_\Omega f(x) dx$ are abbreviated as $\int_\Omega f(x) $.
  \item The sequentiality of the constants $C_i, i = 1, 2, 3, \ldots$ holds only within the lemma/theorem and its proof in which the constants are used. The sequence restarts once the proof is over.
 \end{itemize}
 
 \begin{Remark} System \eqref{model} is a haptotaxis-chemotaxis model with indirect signal production of the haptotactic signal (which is expressed by another population than that performing haptotaxis) and direct degradation of the chemotactic one. There are several models featuring chemotaxis-chemotaxis or chemotaxis-haptotaxis with indirect signal production. Thereby, the tactic behavior is concentrated on one of the interacting populations, see e.g.,  \cite{Chen2024,Mohan2026,surulescu2021does,Wang2023} and can thus be assigned to the multiple taxis models reviewed in \cite{Kolbe2021}, or is distributed among the populations, e.g.,  \cite{Mishra2022,Wu2025,Li2020}. Our model belongs to the latter category, however differing from previous ones not only by the real-world  problem it addresses, but also by the combination of haptotaxis and chemotaxis, which here are, moreover, both of the attractive type. Global existence of solutions is typically ensured for models with indirect signal production, avoiding blow-up often encountered in models where the tactic population is directly expressing its own tactic signal. In this respect our model is no exception.   
 	\end{Remark}

\section{Approximate problems}\label{sec:approx-probl}
In order to construct weak solutions for \eqref{model}, \eqref{boundary_conditions} and \eqref{initial_conditions} by an approximation procedure, we introduce the following regularized problems:
\begin{equation}
  \label{sys1}
  \begin{cases}
    \partial_t c_{1\varepsilon} = a_1\Delta c_{1\varepsilon} - b_{\tau} \nabla \cdot \left( c_{1\varepsilon}  \nabla \tau_\varepsilon\right) - \alpha_1(\chi_\varepsilon)  \frac{c_{1\varepsilon}}{1 + c_{1\varepsilon}}   + \alpha_2(\chi_\varepsilon) \frac{c_{2\varepsilon}}{1 + c_{2\varepsilon}} \\[5pt]
    \hspace{5cm}+ \beta c_{1\varepsilon}(1 - c_{1\varepsilon} - c_{2\varepsilon} - \tau_\varepsilon) - \varepsilon c_{1\varepsilon}^\theta ,    &   x \in \Omega,\  t > 0,  \\[5pt]
    \partial_t c_{2\varepsilon}  = a_2\Delta c_{2\varepsilon} - b_\chi \nabla \cdot \left( c_{2\varepsilon} \nabla \chi_\varepsilon\right) + \alpha_1(\chi_\varepsilon) \frac{c_{1\varepsilon}}{1 + c_{1\varepsilon}}-\alpha_2(\chi_\varepsilon)  \frac{c_{2\varepsilon}}{1 + c_{2\varepsilon}} - \varepsilon c_{2\varepsilon}^\theta,    &    x \in \Omega,\  t > 0,           \\[5pt]
    \partial_t \chi_\varepsilon = D_\chi \Delta \chi_\varepsilon -  a_\chi  (c_{1\varepsilon} + c_{2\varepsilon})\chi_\varepsilon + F_\varepsilon(\chi_\varepsilon),                                        &      x \in \Omega,\ t > 0,           \\[5pt]
    \partial_t \tau_\varepsilon = \varepsilon \Delta \tau_\varepsilon - \delta c_{1\varepsilon}  \tau_\varepsilon - \mu \tau_\varepsilon +  \frac{c_{2\varepsilon}}{1 + c_{2\varepsilon}}, &   x \in \Omega,\ t > 0,\\[5pt]
   \partial_\nu c_{1\varepsilon} = \partial_\nu c_{2\varepsilon} = \partial_\nu \chi_\varepsilon = \partial_\nu \tau_\varepsilon = 0,               &  x \in \partial \Omega,\ t > 0 , \\[5pt]
   c_{1\varepsilon}(x, 0) = c_{10\varepsilon}(x),~ c_{2\varepsilon}(x, 0) = c_{20\varepsilon}(x),~ \chi_{\varepsilon}(x, 0) = \chi_{0\varepsilon}(x),~ \tau_{\varepsilon}(x, 0) = \tau_{0\varepsilon}(x),                     &  x \in \Omega,\ 
  \end{cases}
\end{equation}
for $\varepsilon \in (0, 1)$ and $\theta > \max\{2, n\}$. Here, $F_\varepsilon(\chi_\varepsilon)$ denotes the \emph{mollification} of $F(\chi)$ (see \cite[Appendix C: Calculus, particularly C5]{evans2022partial}), defined by
\begin{align*}
    F_\varepsilon := \eta_\varepsilon \ast F.
\end{align*}
Recalling that
\begin{align*}
     \|F(\chi)\|_{L^\infty(\Omega\times(0,\infty))}
   = \frac{\chi_0}{|\Omega|},
\end{align*}
we conclude that there exists a constant $M_\chi > 0$ such that
\begin{align}
\label{mol_bound_short}
   \|F_\varepsilon(\chi_\varepsilon)\|_{L^\infty(\Omega\times(0,\infty))}
   \le M_\chi.
\end{align}
The families of functions $\{c_{10\varepsilon}\}_{\varepsilon \in (0, 1)}, \{c_{20\varepsilon}\}_{\varepsilon \in (0, 1)}, \{\chi_{0\varepsilon}\}_{\varepsilon \in (0, 1)}$ and $\{\tau_{0\varepsilon}\}_{\varepsilon \in (0, 1)}$ satisfy
\begin{equation}
\label{reg_assumptions}
  \begin{cases}
    c_{10\varepsilon}, c_{20\varepsilon}, \chi_{0\varepsilon}, \tau_{0\varepsilon} \in C^3(\bar{\Omega}),                                                                   \\[3pt]
    c_{10\varepsilon} > 0, c_{20\varepsilon} > 0, \chi_{0\varepsilon} > 0, \tau_{0\varepsilon} > 0 ~\text{in}~ \bar{\Omega},                                                \\[3pt]
    \partial_\nu c_{10\varepsilon} = \partial_\nu c_{20\varepsilon} = \partial_\nu \chi_{0\varepsilon} = \partial_\nu \tau_{0\varepsilon} = 0 ~\text{on}~ \partial \Omega,  \\[3pt]
    c_{10\varepsilon} \to c_{10}, c_{20\varepsilon} \to c_{20},  ~\text{in}~ C^0(\bar{\Omega}) ~\text{as}~ \varepsilon \searrow 0,                       \\[3pt]
    \chi_{0\varepsilon}\to \chi_0 ~\text{and}~ \sqrt{\tau_{0\varepsilon}} \to \sqrt{\tau_0} ~\text{in}~ W^{1, 2}(\Omega) \cap C^0(\bar{\Omega}),  ~\text{as}~ \varepsilon \searrow 0.
  \end{cases}
\end{equation}

\section{Global existence for approximate problems}\label{sec:GE_approx}
\begin{Lemma}
\label{lemma_loc_ext}
Assume that \eqref{alpha_def}, \eqref{F_assump}, and \eqref{reg_assumptions} hold true. Then, for every $\varepsilon \in (0, 1)$, there exists $T_{\max, \varepsilon} \in (0, \infty]$ and a collection of positive functions $c_{1\varepsilon}, c_{2\varepsilon}, \chi_\varepsilon,$ and $\tau_\varepsilon$, each belonging to $C^{2,1}(\bar{\Omega} \times [0, T_{\max, \varepsilon}))$, such that the quadruple $(c_{1\varepsilon}, c_{2\varepsilon}, \chi_\varepsilon, \tau_\varepsilon)$ solves \eqref{sys1} classically in $\Omega \times (0, T_{\max, \varepsilon})$.  Moreover, if $T_{\max, \varepsilon} < \infty$, then for all $\vartheta \in (0, 1)$,
\begin{equation}
\label{exten_criteria}
\limsup_{t \nearrow T_{\max, \varepsilon}} \Big\{\|c_{1\varepsilon}(\cdot, t)\|_{C^{2 + \vartheta}(\bar{\Omega})}
+ \|c_{2\varepsilon}(\cdot, t)\|_{C^{2 + \vartheta}(\bar{\Omega})}
+ \|\chi_{\varepsilon}(\cdot, t)\|_{C^{2 + \vartheta}(\bar{\Omega})}
+ \|\tau_{\varepsilon}(\cdot, t)\|_{C^{2 + \vartheta}(\bar{\Omega})} \Big\} = \infty.
\end{equation}
\end{Lemma}
\begin{proof}
By adapting the arguments from \cite[Lemma 3.1]{stinner2014global}, we can easily prove this result. The non-negativity of solution components can then be established using the strong maximum principle.
\end{proof}

\begin{Lemma}
\label{lemmal1c1}
  For all $\varepsilon \in (0, 1)$ the first solution component of \eqref{sys1} satisfies
  \begin{equation}
    \label{l1c1}
    \int_\Omega c_{1\varepsilon}(\cdot, t) \leq \max\left\{\sup_{\varepsilon \in (0, 1)}\int_\Omega c_{10\varepsilon}, \frac{|\Omega|}{2} \left(1 + \sqrt{\frac{4 M_{\alpha_2}}{\beta}}\right)\right\} =: M_1 \quad \text{for all}~ t \in (0,T_{\max, \varepsilon}).
  \end{equation}
  Additionally, there exists a constant $C > 0$ such that
  \begin{equation}
  \label{c1_bound1}
    \int_t^{t + 1} \int_\Omega c_{1\varepsilon}^2 \leq C \quad \text{for all} ~ t \in (0, T_{\max, \varepsilon} - 1)
  \end{equation}
  and
   \begin{equation}
   \label{c1_bound2}
    \varepsilon \int_t^{t + 1} \int_\Omega c_{1\varepsilon}^\theta \leq C  \quad \text{for all} ~ t \in (0, T_{\max, \varepsilon} - 1).
  \end{equation}
\end{Lemma}

\begin{proof}
Integrating the first equation of \eqref{sys1} over $\Omega$ results in
  \begin{align}
  \label{l1c11}
    \nonumber \dfrac{d}{dt}\int_\Omega c_{1\varepsilon} = & - \int_\Omega \alpha_1(\chi_\varepsilon) \dfrac{c_{1\varepsilon}}{1 + c_{1\varepsilon}} + \int_\Omega \alpha_2(\chi_\varepsilon) \dfrac{c_{2\varepsilon}}{1 + c_{2\varepsilon}} + \beta \int_\Omega c_{1\varepsilon} - \beta \int_\Omega c_{1\varepsilon}^2 - \beta \int_\Omega c_{1\varepsilon} c_{2\varepsilon}  \\[5pt]
    & - \beta \int_\Omega c_{1\varepsilon} \tau_\varepsilon - \varepsilon \int_\Omega c_{1\varepsilon}^\theta  \quad \text{for all}~ t \in (0, T_{\max, \varepsilon}).
  \end{align}
  Using \eqref{alpha_def}, the nonnegativity of solution components, and the fact that $\frac{c_{i\varepsilon}}{1 + c_{i\varepsilon}} < 1$ for $c_{i\varepsilon} \geq 0, i \in \{1, 2\}$, we can have
  \begin{equation*}
    \label{l1c12}
    \dfrac{d}{dt} \int_\Omega c_{1\varepsilon} \leq M_{\alpha_2} |\Omega| + \beta \int_\Omega c_{1\varepsilon} - \beta \int_\Omega c_{1\varepsilon}^2 \quad \text{for all}~ t \in (0, T_{\max, \varepsilon}).
  \end{equation*}
  The Cauchy-Schwartz inequality ensures that $\int_\Omega c_{1\varepsilon}^2 \geq \frac{1}{|\Omega|} (\int_\Omega c_{1\varepsilon})^2$, hence
  \begin{equation}
    \label{l1c13}
    \dfrac{d}{dt} \int_\Omega c_{1\varepsilon} \leq M_{\alpha_2} |\Omega| + \beta \int_\Omega c_{1\varepsilon} - \frac{\beta}{|\Omega|} \left(\int_\Omega c_{1\varepsilon}\right)^2  \quad \text{for all}~ t \in (0, T_{\max, \varepsilon}).
  \end{equation}
 Applying an ODE comparison principle to \eqref{l1c13} results in
  \begin{equation*}
    \|c_{1\varepsilon}(\cdot, t)\|_{L^1(\Omega)} \leq \max\left\{\sup_{\varepsilon \in (0, 1)}\|c_{10\varepsilon}\|_{L^1(\Omega)}, \frac{|\Omega|}{2} \left(1 + \sqrt{\frac{4 M_{\alpha_2}}{\beta}}\right)\right\} \quad \text{for all}~ t \in (0, T_{\max, \varepsilon}),
  \end{equation*}
  thus yielding \eqref{l1c1}. Then \eqref{l1c11} implies
  \begin{align}
  \label{newc1l1}
    \frac{d}{dt}  \int_\Omega c_{1\varepsilon} + \beta \int_\Omega c_{1\varepsilon}^2 + \varepsilon \int_\Omega c_{1\varepsilon}^\theta & \leq M_{\alpha_2} |\Omega| + \beta  \int_\Omega c_{1\varepsilon}   \leq M_{\alpha_2} |\Omega| + \beta M_1\quad \text{for all}~ t \in (0, T_{\max, \varepsilon}).
  \end{align}
  Since $\beta > 0$ and $\varepsilon > 0$ we can derive \eqref{c1_bound1} and \eqref{c1_bound2} by integrating \eqref{newc1l1} over $(t, t + 1)$.

\end{proof}

\begin{Lemma}
\label{lemmal2c2}
For all $\varepsilon \in (0, 1)$, there exists a $C(T) > 0$ such that the second solution component of \eqref{sys1} satisfies
  \begin{equation}
    \label{l1c2}
    \int_\Omega c_{2\varepsilon}(\cdot, t) \leq C(T) \quad \text{for all}~ t \in (0, \widehat{T}_\varepsilon),
  \end{equation}
  and
   \begin{equation}
   \label{c2_bound2}
    \varepsilon \int_0^{\widehat{T}_\varepsilon} \int_\Omega c_{2\varepsilon}^\theta \leq C(T)
  \end{equation}
  where $\widehat{T}_\varepsilon := \min\{T, T_{\max, \varepsilon}\}$ for $T > 0$.
\end{Lemma}

\begin{proof}
Integrating the second equation of \eqref{sys1} over $\Omega$ yields
  \begin{equation}
  \label{l1c21}
    \dfrac{d}{dt}\int_\Omega c_{2\varepsilon} = \int_\Omega \alpha_1(\chi_\varepsilon) \dfrac{c_{1\varepsilon}}{1 + c_{1\varepsilon}} - \int_\Omega \alpha_2(\chi_\varepsilon) \dfrac{c_{2\varepsilon}}{1 + c_{2\varepsilon}} - \varepsilon \int_\Omega c_{2\varepsilon}^\theta \quad \text{for all}~ t \in (0, T_{\max, \varepsilon}).
  \end{equation}
  Using \eqref{alpha_def}, the non-negativity of solution components, and the fact that $\frac{c_{i\varepsilon}}{1 + c_{i\varepsilon}} < 1$ for $c_{i\varepsilon} \geq 0, i \in \{1, 2\}$, we can have
  \begin{equation}
  \label{l1c22}
    \dfrac{d}{dt}\int_\Omega c_{2\varepsilon} \leq M_{\alpha_1}|\Omega| \quad \text{for all}~ t \in (0, T_{\max, \varepsilon})
  \end{equation}
    Integrating \eqref{l1c22} over the time interval $(0, \widehat{T}_\varepsilon)$  results in
   \begin{equation*}
     \|c_{2\varepsilon}(\cdot, t)\|_{L^1(\Omega)} \leq \max \left\{\sup_{\varepsilon \in (0, 1)} \int_\Omega c_{20\varepsilon} + M_{\alpha_1}|\Omega|\widehat{T}_\varepsilon\right\}
   \end{equation*}
   thus yielding \eqref{l1c2}. In view of \eqref{l1c2}, integrating \eqref{l1c21} over  $(0, \widehat{T}_\varepsilon)$ yields \eqref{c2_bound2}.
\end{proof}


\begin{Lemma}
    \label{lemma_chi}
    For all $\varepsilon \in (0, 1)$, there exists a $C(T) > 0$ and $\chi_\infty > 0$ such that the third solution component of \eqref{sys1} satisfies
  \begin{equation}
    \label{l1c2chi}
    \int_\Omega \chi_{\varepsilon}(\cdot, t) \leq C(T) \quad \text{for all}~ t \in (0, \widehat{T}_\varepsilon),
  \end{equation}
  and
   \begin{equation}
   \label{chi_bound2}
   \|\chi(\cdot, t)\|_{L^\infty(\Omega)} \leq \chi_\infty \quad \text{for all}~ t \in (0, \widehat{T}_\varepsilon)
  \end{equation}
  where $\widehat{T}_\varepsilon := \min\{T, T_{\max, \varepsilon}\}$ for $T > 0$.
\end{Lemma}
\begin{proof}
The estimate in \eqref{l1c2chi} follows by integrating the third equation in \eqref{sys1} over space and using the nonnegativity of the solution components together with \eqref{mol_bound_short}, which yields
\begin{align}
    \label{uzr}
    \frac{d}{dt} \int_\Omega \chi_\varepsilon \leq M_\chi |\Omega| \quad \text{for all}~ t \in (0, T_{\max, \varepsilon}).
\end{align}
Integrating \eqref{uzr} over the time interval $(0, \widehat{T}_\varepsilon)$ gives
\begin{align*}
    \|\chi_\varepsilon(\cdot, t)\|_{L^1(\Omega)} \leq \max \Bigg\{ \sup_{\varepsilon \in (0, 1)} \int_\Omega \chi_{0\varepsilon}, \; M_{\chi} |\Omega| \widehat{T}_\varepsilon \Bigg\},
\end{align*}
which yields the desired estimate.\\[-1ex]

\noindent
Applying the variation-of-constants formula to the third equation of \eqref{sys1} and using the properties for the Neumann heat semigroup \cite[Lemma 1.3]{winkler2010aggregation} and \eqref{mol_bound_short}, we estimate
\begin{align}
  \label{c2_bound1}
    \nonumber \|\chi_{\varepsilon}(\cdot, t)\|_{L^\infty(\Omega)} & = \sup_{\Omega} \bigg(e^{D_\chi t \Delta }\chi_{0\varepsilon} - \int_0^t e^{(t - s) D_\chi \Delta} \left\{a_\chi  (c_{1\varepsilon}(\cdot, s) + c_{2\varepsilon}(\cdot, s))\chi_\varepsilon(\cdot, s)\right\} ds + \int_0^t e^{(t - s) D_\chi \Delta} F_\varepsilon(\chi_\varepsilon(\cdot, s)) ds \bigg)  \\[5pt]
    \nonumber & \leq \sup_{\Omega} \left(e^{D_\chi t \Delta }\chi_{0\varepsilon}\right) + \sup_{\Omega} \left( \int_0^t e^{(t - s) D_\chi \Delta} F_\varepsilon(\chi_\varepsilon(\cdot, s)) ds\right)  \\[5pt]
    \nonumber & \leq\|e^{D_\chi t \Delta }\chi_{0\varepsilon}\|_{L^\infty(\Omega)} +  \int_0^t \|e^{(t - s) D_\chi \Delta} F_\varepsilon(\chi_\varepsilon(\cdot, s))\|_{L^\infty(\Omega)}  ds  \\[5pt]
    \nonumber & \leq \|e^{D_\chi t \Delta }\chi_{0\varepsilon}\|_{L^\infty(\Omega)} + M_{\chi} C_1 \widehat{T}_\varepsilon \qquad \text{for all}~ t \in (0, \widehat{T}_\varepsilon).
  \end{align}
From the above, we see that there exists a constant $\chi_\infty > 0$ such that
\begin{equation*}
  \|\chi_{\varepsilon}(\cdot, t)\|_{L^\infty(\Omega)} \leq \chi_\infty  \qquad \text{for all}~ t \in (0, \widehat{T}_\varepsilon).
\end{equation*}
\end{proof}
\begin{Lemma}
  \label{ht_bound}
   For all $\varepsilon \in (0, 1)$,  the fourth solution component of \eqref{sys1} satisfies
  \begin{align}
    \label{t_bound} & \|\tau_\varepsilon(\cdot, t)\|_{L^\infty(\Omega)} \leq \frac{r_\ast}{\mu} + \|\tau_{0 \varepsilon }\|_{L^\infty(\Omega)} =: \tau_\ast,\quad \text{for all}~ t \in (0, T_{\max, \varepsilon})
  \end{align}
   where $r\varepsilon:= \frac{c_{2\varepsilon}}{1 + c_{2\varepsilon}}$ and $r_\ast := \|r_\varepsilon\|_{L^\infty(\Omega \times (0, \infty))}$.
\end{Lemma}

\begin{proof}
  Let
  \begin{align*}
    \bar{\tau}(t) := \|\tau_{0 \varepsilon}\|_{L^\infty(\Omega)}e^{- \mu t} + \int_0^t e^{-\mu (t - s)} \|r_\varepsilon (\cdot, s)\|_{L^\infty(\Omega)} ds
  \end{align*}
  for $t \in (0, \infty)$. Then $\bar{\tau} \leq \tau_\ast$ on $\bar{\Omega} \times (0, T_{\max, \varepsilon})$ and the following holds true
  \begin{equation*}
    \bar{\tau}_t - \varepsilon \Delta \bar{\tau} + \delta c_{1 \varepsilon} \bar{\tau} + \mu \bar{\tau} - r(x, t) \geq \bar{\tau}_t + \mu \bar{\tau} - \|r(\cdot, t)\|_{L^\infty(\Omega)} \geq 0
  \end{equation*}
  on $\bar{\Omega} \times (0, T_{\max, \varepsilon})$, owing to the non-negativity of $c_{1\varepsilon}$ and $\tau_\varepsilon$. Applying an ODE comparison principle results in $\bar{\tau} \geq \tau$ on $(\Omega \times (0, T_{\max, \varepsilon})$ and, in particular $\|\tau_\varepsilon(\cdot, t)\|_{L^\infty(\Omega)} \leq \tau_\ast$ for all $t \in (0, T_{\max, \varepsilon})$. This gives \eqref{t_bound}.
\end{proof}

\noindent
With these estimates in hand, we now show that, for any fixed $\varepsilon \in (0, 1)$, the solution $(c_{1\varepsilon}, c_{2\varepsilon}, \chi_\varepsilon, \tau_\varepsilon)$ obtained in Lemma \ref{lemma_loc_ext} exists globally in time.
\begin{Lemma}[Global existence for \eqref{sys1}]
   For all $\varepsilon \in (0, 1)$, the solution $(c_{1\varepsilon}, c_{2\varepsilon}, \chi_\varepsilon, \tau_\varepsilon)$ of \eqref{sys1} obtained in Lemma \ref{lemma_loc_ext} is global, i.e. $T_{\max, \varepsilon} = \infty$.
\end{Lemma}
\begin{proof}
  We will establish $T_{\max, \varepsilon} = \infty$ by contradiction. Fix $\varepsilon \in (0, 1)$ and suppose that $T_{\max, \varepsilon} < \infty$. By Lemmas \ref{lemmal1c1} and \ref{lemmal2c2} we can have a $C(\varepsilon, T) > 0$ such that
  \begin{equation}
  \label{neq1}
    \int_0^{\widehat{T}_\varepsilon} \int_\Omega c_{1\varepsilon}^\theta \leq C(\varepsilon, T), \quad ~\text{and}~ \int_0^{\widehat{T}_\varepsilon} \int_\Omega c_{2\varepsilon}^\theta \leq C(\varepsilon, T),
  \end{equation}
  where $\widehat{T}_\varepsilon := \min\{T, T_{\max, \varepsilon}\}$ for $T > 0$. Using \eqref{mol_bound_short}, \eqref{chi_bound2} and \eqref{neq1}, we can have the boundedness of $f_\varepsilon := \partial_t \chi_\varepsilon - D_\chi \chi_\varepsilon \equiv - a_\chi(c_{1 \varepsilon} + c_{2 \varepsilon}) \chi_\varepsilon + F_\varepsilon(\chi_\varepsilon)$ in $L^\theta(\Omega \times (0, \widehat{T}_\varepsilon))$. Applying the maximal Sobolev regularity for parabolic equations \cite{matthias1997heat} to the third equation of \eqref{sys1}, we can have a $C_1(\varepsilon, T) > 0$ such that
  \begin{equation}
  \label{ge1}
    \int_0^{\widehat{T}_\varepsilon} \|\chi_\varepsilon(\cdot, t)\|^\theta_{W^{2, \theta}(\Omega)} dt \leq C_1(\varepsilon, T).
  \end{equation}
 Similarly, $g_\varepsilon := \partial_t \tau_\varepsilon - \varepsilon \Delta \tau_\varepsilon \equiv - \delta_1 c_{1\varepsilon} \tau_\varepsilon - \mu \tau_\varepsilon + \frac{c_{2\varepsilon}}{1 + c_{2\varepsilon}}$ also belongs to $L^\theta(\Omega \times (0, \widehat{T}_\varepsilon))$, thanks to \eqref{c1_bound2}, \eqref{t_bound} and the fact that $\frac{c_{2\varepsilon}}{1 + c_{2\varepsilon}} < 1$ for $c_{2\varepsilon} \geq 0$. Applying the same  regularity to the fourth equation in \eqref{sys1}, yields a $C_2(\varepsilon, T) > 0$ fulfilling
  \begin{equation}
  \label{ge2}
    \int_0^{\widehat{T}_\varepsilon} \|\tau_\varepsilon(\cdot, t)\|^\theta_{W^{2, \theta}(\Omega)} dt \leq C_2(\varepsilon, T).
  \end{equation}
  Since $\theta > \max\{2, n\}$ we can apply the Sobolev embedding $W^{2, \theta}(\Omega) \hookrightarrow W^{1, \infty}(\Omega)$ along with H\"{o}lder's inequality to have a $\widehat{C}(\varepsilon, T) > 0$ such that
  \begin{equation}
  \label{ge3}
    \int_0^{\widehat{T}_\varepsilon} \|\nabla \chi_\varepsilon(\cdot, t)\|^2_{L^\infty(\Omega)} dt \leq \widehat{C}(\varepsilon, T), \quad \text{and}~ \int_0^{\widehat{T}_\varepsilon} \|\nabla \tau_\varepsilon(\cdot, t)\|^2_{L^\infty(\Omega)} dt \leq \widehat{C}(\varepsilon, T).
  \end{equation}
 Now we will apply a standard testing procedure to the first and second equations in \eqref{sys1}. To this end we  first multiply the second equation in \eqref{sys1} by $c_{2 \varepsilon}^{p - 1} (p > 1)$, and integrate by parts to have
\begin{align}
\label{ge4}
   \frac{1}{p} \frac{d}{dt} \int_\Omega c_{2\varepsilon}^p + (p - 1) a_2  \int_\Omega c_{2\varepsilon}^{p - 2}|\nabla c_{2\varepsilon}|^2 \leq (p - 1) b_\chi \int_\Omega c_{2\varepsilon}^{p - 1} \nabla c_{2\varepsilon} \cdot \nabla \chi_\varepsilon + \int_\Omega \alpha_1(\chi_\varepsilon) \dfrac{c_{1\varepsilon} }{1 + c_{1\varepsilon}} c_{2\varepsilon}^{p - 1}
\end{align}
for all $t \in (0, T_{\max, \varepsilon})$. Applying Young's inequality to the two terms on the right-hand side in \eqref{ge4} gives us
\begin{align}
  \label{ge5}
   & (p - 1) b_\chi \int_\Omega c_{2\varepsilon}^{p - 1} \nabla c_{2\varepsilon} \cdot \nabla \chi_\varepsilon \leq \frac{(p - 1) a_2}{2}\int_\Omega c_{2\varepsilon}^{p - 2} |\nabla c_{2 \varepsilon}|^2 + \frac{(p - 1) b_\chi^2}{2 a_2} \int_\Omega c_{2\varepsilon}^p |\nabla \chi_\varepsilon|^2  \\[5pt]
  \label{ge6}
  & \int_\Omega \alpha_1(\chi_\varepsilon) \dfrac{c_{1\varepsilon} }{1 + c_{1\varepsilon}} c_{2\varepsilon}^{p - 1}  \leq M_{\alpha_1} \int_\Omega c_{2\varepsilon}^{p - 1} \leq \int_\Omega c_{2\varepsilon}^p + M_{\alpha_1}^p |\Omega|.
\end{align}
Inserting \eqref{ge5} and \eqref{ge6} in \eqref{ge4} yields
\begin{align}
\label{c2ineq1}
\frac{d}{dt} \int_\Omega c_{2\varepsilon}^p + \frac{p(p - 1) a_2}{2}\int_\Omega c_{2\varepsilon}^{p - 2} |\nabla c_{2 \varepsilon}|^2\leq \frac{p (p - 1) b_\chi^2 }{2 a_2} \int_\Omega c_{2\varepsilon}^p |\nabla \chi_\varepsilon|^2 + p \int_\Omega c_{2\varepsilon}^p + p M_{\alpha_1}^p|\Omega|
\end{align}
for all $t \in (0, T_{\max, \varepsilon})$. From \eqref{c2ineq1} we can have
\begin{align}
  \label{c2p_bound1}
  \frac{d}{dt} \int_\Omega c_{2\varepsilon}^p \leq \left(\frac{p (p - 1) b_\chi^2 }{2 a_2} \|\nabla \chi_\varepsilon(\cdot, t)\|^2_{L^\infty(\Omega)} + p \right)\int_\Omega c_{2\varepsilon}^p + p M_{\alpha_1}^p|\Omega|
\end{align}
for all $t \in (0, T_{\max, \varepsilon})$. In view of \eqref{ge3}, applying Gronwall's inequality to \eqref{c2p_bound1} results in
\begin{align}
  \label{c2p_bound}
  \nonumber \int_\Omega c^p_{2\varepsilon}(\cdot, t) & \leq \exp\left\{\frac{p(p - 1) b_\chi^2}{2 a_2} \int_0^t \|\nabla \chi_\varepsilon(\cdot, s)\|^2_{L^\infty(\Omega)}ds + pt\right\}\left(\int_\Omega c_{20\varepsilon}^p + p M_{\alpha_1}^p |\Omega|t\right)   \\[5pt]
  & \leq C_3 \quad \text{for all}~ t \in (0, \widehat{T}_\varepsilon).
\end{align}
\noindent
Now, we multiply the first equation in \eqref{sys1} by $c_{1 \varepsilon}^{p - 1}$, integrate by parts, and drop a few negative terms to have
\begin{align}
\label{c1ge4}
  \frac{1}{p} \frac{d}{dt} \int_\Omega c_{1\varepsilon}^p & + (p - 1) a_1  \int_\Omega c_{1\varepsilon}^{p - 2}|\nabla c_{1\varepsilon}|^2 \leq (p - 1) b_\tau \int_\Omega c_{1\varepsilon}^{p - 1}\nabla c_{1\varepsilon} \cdot \nabla \tau_\varepsilon + \int_\Omega \alpha_2(\chi_\varepsilon) \dfrac{c_{2\varepsilon} }{1 + c_{2\varepsilon}} c_{1\varepsilon}^{p - 1} + \beta \int_\Omega c_{1\varepsilon}^p
\end{align}
for all $t \in (0, T_{\max, \varepsilon})$. Applying Young's inequality to the first two terms on the right-hand side in \eqref{c1ge4} gives us
\begin{align}
  \label{c1ge5}
  & (p - 1) b_\tau \int_\Omega c_{1\varepsilon}^{p - 1} \nabla c_{1\varepsilon} \cdot \nabla \tau_\varepsilon\leq \frac{(p - 1) a_1}{4} \int_\Omega c_{1\varepsilon}^{p - 2} |\nabla c_{1 \varepsilon}|^2 + \frac{(p - 1) b_\tau^2}{ a_1} \int_\Omega c_{1\varepsilon}^p |\nabla \tau_\varepsilon|^2, \\
  \label{c1ge7}
  & \int_\Omega \alpha_2(\chi_\varepsilon) \dfrac{c_{2\varepsilon} }{1 + c_{2\varepsilon}} c_{1\varepsilon}^{p - 1}  \leq M_{\alpha_2} \int_\Omega c_{1\varepsilon}^{p - 1} \leq \int_\Omega c_{1\varepsilon}^p + M_{\alpha_2}^p |\Omega|.
\end{align}
Inserting \eqref{c1ge5}-\eqref{c1ge7} in \eqref{c1ge4} yields
\begin{align}
  \label{c1ineq}
   \frac{d}{dt} \int_\Omega c_{1\varepsilon}^p \leq  \frac{p(p - 1) b_\tau^2 }{ a_1} \int_\Omega c_{1\varepsilon}^p |\nabla \tau_\varepsilon|^2 + p(\beta + 1) \int_\Omega c_{1\varepsilon}^p + pM_{\alpha_2}^p |\Omega|
  \end{align}
for all $t \in (0, T_{\max, \varepsilon})$. From \eqref{c1ineq} we can have
\begin{align}
  \label{c1ineq1}
  \frac{d}{dt} \int_\Omega c_{1\varepsilon}^p \leq & \left\{\frac{p(p - 1) b_\tau^2}{a_1} \|\nabla \tau_\varepsilon(\cdot, t)\|^2_{L^\infty(\Omega)} + p (\beta + 1)\right\} \int_\Omega c_{1\varepsilon}^p + p M_{\alpha_2}^p |\Omega|
\end{align}
for all $t \in (0, T_{\max, \varepsilon})$. Applying Gronwall's inequality to \eqref{c1ineq1} in view of \eqref{ge3} results in
\begin{align}
  \label{c1pbound}
  \int_\Omega c^p_{1\varepsilon}(\cdot, t) \leq & \exp\bigg\{\frac{p(p - 1)b_\tau^2}{a_1} \int_0^t \|\nabla \tau_\varepsilon(\cdot, s)\|_{L^\infty(\Omega)}^2 ds  + p(\beta + 1)t\bigg\} \left(\int_\Omega c_{10\varepsilon} + p M_{\alpha_2}^p |\Omega| t\right) \leq C_4
\end{align}
for all $t \in (0, \widehat{T}_\varepsilon)$. From \eqref{mol_bound_short}, \eqref{chi_bound2}, \eqref{t_bound}, \eqref{c2p_bound1}, and \eqref{c1pbound}, we deduce that
\begin{align*}
& f_\varepsilon := a_\chi (c_{1 \varepsilon} + c_{2 \varepsilon}) \chi_\varepsilon + F_\varepsilon(\chi_\varepsilon), \\[5pt]
& g_\varepsilon := - \delta_1 c_{1\varepsilon} \tau_\varepsilon - \mu \tau_\varepsilon + \tfrac{c_{2\varepsilon}}{1 + c_{2\varepsilon}},
\end{align*}
belong to $L^\infty((0, \widehat{T}_\varepsilon); L^p(\Omega))$ for any $p > 1$. Using this together with the variation-of-constants formula and the properties of Neumann heat semigroups \cite[Lemma 1.3]{winkler2010aggregation}, we obtain the following estimates:
\begin{align}
\label{nabchi_bound}
  \nonumber & \|\nabla \chi_\varepsilon(\cdot, t)\|_{L^\infty(\Omega)} \leq \|\nabla e^{D_\chi t \Delta}\chi_\varepsilon(\cdot, 0)\|_{L^\infty(\Omega)} + \int_0^t \| \nabla e^{D_\chi (t - s)\Delta }f_\varepsilon  \|_{L^\infty(\Omega)}ds\\
    & \leq C_5(1 + (D_\chi t)^{-\frac{1}{2}})e^{-\lambda_1 D_\chi t} \|\chi_{0\varepsilon}\|_{L^\infty(\Omega)}
   + \int_0^t C_5(1 + (D_\chi (t - s))^{-\frac{1}{2} - \frac{n}{2}\frac{1}{p}})e^{- \lambda_1 D_\chi(t - s)}\|f_\varepsilon \|_{L^p(\Omega)}ds
\end{align}
and
\begin{align}
\label{nabtau_bound}
  \nonumber & \|\nabla \tau_\varepsilon(\cdot, t)\|_{L^\infty(\Omega)} \leq \|\nabla e^{\varepsilon t \Delta}\tau_\varepsilon(\cdot, 0)\|_{L^\infty(\Omega)} + \int_0^t \| \nabla e^{\varepsilon (t - s)\Delta }g_\varepsilon  \|_{L^\infty(\Omega)}ds\\
    & \leq C_6(1 + (\varepsilon t)^{-\frac{1}{2}})e^{-\lambda_1 \varepsilon t} \|\tau_{0\varepsilon}\|_{L^\infty(\Omega)}
   + \int_0^t C_6(1 + (\varepsilon (t - s))^{-\frac{1}{2} - \frac{n}{2}\frac{1}{p}})e^{- \lambda_1 \varepsilon (t - s)}\|g_\varepsilon \|_{L^p(\Omega)}ds,
\end{align}
for all $t \in (0, \widehat{T}_{\varepsilon})$ and all $p \in (1, \infty)$. Here, $\lambda_1 > 0$ denotes the first nonzero eigenvalue of the Laplacian with Neumann boundary conditions in $\Omega$. Choosing $p \in (n, \infty)$ and using the boundedness of $\|f_\varepsilon(\cdot, s)\|_{L^p(\Omega)}$ and $\|g_\varepsilon(\cdot, s)\|_{L^p(\Omega)}$ for $s \in (0, \widehat{T}_\varepsilon)$, the integrals in \eqref{nabchi_bound} and \eqref{nabtau_bound} are seen to converge.\\[-1ex]

\noindent
The above estimates enable us to establish global existence in time for \eqref{sys1}. To this end we define the following four functions:
\begin{align*}
  & A_{c_1}(x, t, c_{1\varepsilon}, q) := a_1 q - b_\tau c_{1\varepsilon} \nabla \tau_\varepsilon(x, t),          \\[5pt]
  & B_{c_1}(x, t, c_{1\varepsilon}, c_{2\varepsilon}) := - \alpha_1(\chi_\varepsilon)  \tfrac{c_{1\varepsilon}}{1 + c_{1\varepsilon}}   + \alpha_2(\chi_\varepsilon) \tfrac{c_{2\varepsilon}}{1 + c_{2\varepsilon}}
    + \beta c_{1\varepsilon}(1 - c_{1\varepsilon} - c_{2\varepsilon} - \tau_\varepsilon) - \varepsilon c_{1\varepsilon}^\theta,  \\[5pt]
  & A_{c_2}(x, t, c_{2\varepsilon}, q) := a_2 q - b_\chi c_{2\varepsilon} \nabla \chi_\varepsilon(x, t), \\[5pt]
  & B_{c_2}(x, t, c_{2\varepsilon}, c_{2\varepsilon}) :=  \alpha_1(\chi_\varepsilon) \tfrac{c_{1\varepsilon}}{1 + c_{1\varepsilon}}-\alpha_2(\chi_\varepsilon)  \tfrac{c_{2\varepsilon}}{1 + c_{2\varepsilon}} - \varepsilon c_{2\varepsilon}^\theta,
\end{align*}
where  $A_{c_i} \in \Omega \times (0, T_{\max, \varepsilon}) \times \mathbb{R} \times \mathbb{R}^n$ and $B_{c_i} \in \Omega \times (0, T_{\max, \varepsilon}) \times \mathbb{R} \times \mathbb{R}, i \in \{1, 2\}$ respectively. \\[-1ex]

\noindent
From \eqref{sys1} we can easily observe that
\begin{align*}
  &\partial_t c_{1\varepsilon} - \nabla \cdot (A_{c_1}(x, t, c_{1\varepsilon}, \nabla c_{1\varepsilon})) = B_{c_1}(x, t, c_{1\varepsilon}, \nabla c_{1\varepsilon}),   \quad \text{and}\\[5pt]
  &\partial_t c_{2\varepsilon} - \nabla \cdot (A_{c_2}(x, t, c_{1\varepsilon}, \nabla c_{1\varepsilon})) = B_{c_2}(x, t, c_{1\varepsilon}, \nabla c_{1\varepsilon})
\end{align*}
hold for all $(x, t) \in \Omega \times (0, T_{\max, \varepsilon}) $. For all $(x, t) \in \Omega \times (0, T_{\max, \varepsilon})$ we can have
\begin{equation*}
  \begin{cases}
     A_{c_1}(x, t, c_{1\varepsilon}, \nabla c_{1\varepsilon})\nabla c_{1\varepsilon}  \geq \frac{a_1}{2} |\nabla c_{1\varepsilon}|^2 - \psi_1(x, t),  \\[5pt]
    |A_{c_1}(x, t, c_{1\varepsilon}, \nabla c_{1\varepsilon})| \leq a_1 |\nabla c_{1\varepsilon}| + \psi_2(x, t),   \\[5pt]
    |B_{c_1}(x, t, c_{1\varepsilon}, \nabla c_{1\varepsilon}))| \leq \psi_3(x, t),
  \end{cases}
\end{equation*}
\begin{equation*}
  \begin{cases}
     A_{c_2}(x, t, c_{2\varepsilon}, \nabla c_{2\varepsilon})\nabla c_{2\varepsilon}  \geq \frac{a_2}{2} |\nabla c_{1\varepsilon}|^2 - \psi_4(x, t),  \\[5pt]
    |A_{c_2}(x, t, c_{2\varepsilon}, \nabla c_{2\varepsilon})| \leq a_2 |\nabla c_{2\varepsilon}| + \psi_5(x, t),   \\[5pt]
    |B_{c_2}(x, t, c_{2\varepsilon}, \nabla c_{2\varepsilon}))| \leq \psi_6(x, t),
  \end{cases}
\end{equation*}
\noindent
where
\begin{align*}
 \psi_1 &:= \frac{b_\tau^2}{a_1} c_{1\varepsilon}^2 |\nabla \tau_\varepsilon|^2,\\ \psi_2 &:=  b_1 c_{1\varepsilon} |\nabla \tau_\varepsilon| + b_2 c_{1\varepsilon} |\nabla \chi_\varepsilon|,\\
  \psi_3 &:= \alpha_1(\chi_\varepsilon)  c_{1\varepsilon}  + \alpha_2(\chi_\varepsilon) c_{2\varepsilon} + \beta c_{1\varepsilon} + \beta c_{1\varepsilon}^2 + \beta c_{1\varepsilon} c_{2\varepsilon} + \beta c_{1\varepsilon} \tau_\varepsilon +\varepsilon c_{1\varepsilon}^\theta,\\
  \psi_4 &:=  \frac{b_\chi^2}{2a_2} c_{2\varepsilon}^2 |\nabla \chi_\varepsilon|^2,\\
 \psi_5 &:= b_\chi c_{2\varepsilon} |\nabla \chi_\varepsilon|,\\ 
 \psi_6 &:= \alpha_1(\chi_\varepsilon) c_{1\varepsilon}+\alpha_2(\chi_\varepsilon)  c_{2\varepsilon} + \varepsilon c_{2\varepsilon}^\theta
 \end{align*}
  are nonnegative functions in $\Omega \times (0, T_{\max, \varepsilon})$. \\[-1.5ex]
  
  \noindent
  If we set $T = T_{\max, \varepsilon}$ then by  \eqref{c2p_bound}, \eqref{c1pbound}, \eqref{nabchi_bound} and \eqref{nabtau_bound} all the $\psi_i's, i = \{1, 2 \ldots, 6\}$ belong to $L^p(\Omega \times (0, T_{\max, \varepsilon}))$ for every $p > 1$. Applying the parabolic H\"{o}lder estimates \cite[Theorem 1.3 and Remark 1.4]{ladyzhenskaia1968linear}, we can claim that there exists a $\vartheta \in (0, 1)$ fulfilling
\begin{equation*}
  \|c_{1\varepsilon}\|_{C^{\vartheta, \frac{\vartheta}{2}}(\bar{\Omega} \times [0, T_{\max, \varepsilon}])} \leq C_5 \quad \text{and}~ \|c_{2\varepsilon}\|_{C^{\vartheta, \frac{\vartheta}{2}}(\bar{\Omega} \times [0, T_{\max, \varepsilon}])} \leq C_5
\end{equation*}
with $C_5 > 0$. Applying the standard parabolic Schauder estimates \cite{ladyzhenskaia1968linear} to $\chi_\varepsilon$ and $\tau_\varepsilon$ equations in \eqref{sys1}, we can find a $C_6 > 0$ such that
\begin{equation}
\label{contra1}
  \|\chi_{\varepsilon}\|_{C^{2 + \vartheta, 1 + \frac{\vartheta}{2}}(\bar{\Omega} \times [0, T_{\max, \varepsilon}])} \leq C_6 \quad \text{and}~ \|\tau_{\varepsilon}\|_{C^{2 + \vartheta , 1 + \frac{\vartheta}{2}}(\bar{\Omega} \times [0, T_{\max, \varepsilon}])} \leq C_6.
\end{equation}
Again, the standard parabolic Schauder estimates enable us to find a $C_7 > 0$ such that
\begin{equation}
\label{contra2}
  \|c_{1\varepsilon}\|_{C^{2 + \vartheta, 1 + \frac{\vartheta}{2}}(\bar{\Omega} \times [0, T_{\max, \varepsilon}])} \leq C_7 \quad \text{and}~ \|c_{2\varepsilon}\|_{C^{2 + \vartheta , 1 + \frac{\vartheta}{2}}(\bar{\Omega} \times [0, T_{\max, \varepsilon}])} \leq C_7.
\end{equation}
Taken together, \eqref{contra1} and \eqref{contra2} contradict the extensibility criterion \eqref{exten_criteria}.

\end{proof}

\section{An entropy-type functional}\label{sec:entropy}
This section aims to derive some estimates that stem from an entropy-type functional and are the main step towards the existence of a global weak solution to \eqref{model}. We will initially establish certain inequalities that will prove to be valuable later on.
\begin{Lemma}
  \label{entropy1}
  There exists a $C > 0$ such that any solution of \eqref{sys1} satisfies
  \begin{align}
    \label{c1entro_ineq1}
    \nonumber \frac{d}{dt} \int_\Omega c_{1\varepsilon} \ln c_{1\varepsilon}  + a_1 \int_\Omega \frac{|\nabla c_{1\varepsilon}|^2}{c_{1\varepsilon}} & + \frac{\beta}{2} \int_\Omega c_{1\varepsilon}^2 \ln(2 + c_{1\varepsilon}) + \frac{\varepsilon}{2} \int_\Omega c_{1\varepsilon}^\theta \ln(2 + c_{1\varepsilon}) \leq b_\tau \int_\Omega  \nabla \tau_\varepsilon \cdot \nabla c_{1\varepsilon} \\[5pt]
    & +  M_{\alpha_2}\int_{\{ c_{1\varepsilon} > 1\}} \ln c_{1\varepsilon} + \beta \int_\Omega c_{1\varepsilon}  + \frac{\beta}{e} \int_\Omega c_{2\varepsilon}  + C
  \end{align}
  for all $\varepsilon \in (0, 1)$ and all $t > 0$.
 \end{Lemma}
\begin{proof}
  We use the first equation of \eqref{sys1}, the non-negativity of solution components, and integration by parts to calculate, after testing with $\ln c_{1\varepsilon}$,
  \begin{align}
  \label{cient11}
    \nonumber & \frac{d}{dt} \int_\Omega c_{1\varepsilon} \ln c_{1\varepsilon}  = \int_\Omega \partial_t c_{1\varepsilon} \cdot \ln c_{1\varepsilon} + \int_\Omega \partial_t c_{1\varepsilon}   \\[5pt]
    \nonumber &=  a_1 \int_\Omega \Delta c_{1\varepsilon} \cdot \ln c_{1\varepsilon}  - b_{\tau}\int_\Omega \nabla \cdot \left(c_{1\varepsilon}  \nabla \tau_\varepsilon\right) \ln c_{1\varepsilon} - \int_\Omega \alpha_1(\chi_\varepsilon)  \tfrac{c_{1\varepsilon}}{1 + c_{1\varepsilon}} \ln c_{1\varepsilon} + \int_\Omega  \alpha_2(\chi_\varepsilon) \tfrac{c_{2\varepsilon}}{1 + c_{2\varepsilon}}   \ln c_{1\varepsilon}  \\[5pt]
    \nonumber &  + \beta \int_\Omega c_{1\varepsilon}(1 - c_{1\varepsilon} -  c_{2\varepsilon} - \tau_\varepsilon)\ln c_{1\varepsilon}  - \varepsilon \int_\Omega c_{1\varepsilon}^\theta \ln c_{1 \varepsilon}  - \int_\Omega \alpha_1(\chi_\varepsilon) \tfrac{c_{1\varepsilon}}{1 + c_{1\varepsilon}} + \int_\Omega \alpha_2(\chi_\varepsilon) \tfrac{c_{2\varepsilon}}{1 + c_{2\varepsilon}} \\[5pt]
    \nonumber &  + \beta \int_\Omega c_{1\varepsilon}(1 - c_{1\varepsilon} -  c_{2\varepsilon} - \tau_\varepsilon) - \varepsilon \int_\Omega c_{1\varepsilon}^\theta \\[5pt]
    \nonumber & \leq - a_1 \int_\Omega \frac{|\nabla c_{1\varepsilon}|^2}{c_{1\varepsilon}} + b_\tau \int_\Omega  \nabla \tau_\varepsilon \cdot \nabla c_{1\varepsilon} - \int_\Omega \alpha_1(\chi_\varepsilon)  \tfrac{c_{1\varepsilon}}{1 + c_{1\varepsilon}} \cdot \ln c_{1\varepsilon} + \int_\Omega  \alpha_2(\chi_\varepsilon) \tfrac{c_{2\varepsilon}}{1 + c_{2\varepsilon}}  \cdot \ln c_{1\varepsilon}  \\[5pt]
    & \nonumber   + \beta \int_\Omega c_{1\varepsilon}(1 - c_{1\varepsilon} -  c_{2\varepsilon} - \tau_\varepsilon)\ln c_{1\varepsilon} - \varepsilon \int_\Omega c_{1\varepsilon}^\theta \ln c_{1 \varepsilon}  - \int_\Omega \alpha_1(\chi_\varepsilon) \tfrac{c_{1\varepsilon}}{1 + c_{1\varepsilon}} + \int_\Omega \alpha_2(\chi_\varepsilon) \tfrac{c_{2\varepsilon}}{1 + c_{2\varepsilon}} \\[5pt]
    & + \beta \int_\Omega c_{1\varepsilon}(1 - c_{1\varepsilon} -  c_{2\varepsilon} - \tau_\varepsilon)  - \varepsilon \int_\Omega c_{1\varepsilon}^\theta  \quad \text{for all}~ t > 0.
  \end{align}
 Utilizing \eqref{alpha_def},  the non-negativity of the solution components, \eqref{t_bound}, the identity $- s \ln s \leq \frac{1}{e}$ for $s > 0$, and the fact that $\frac{c_{i\varepsilon}}{1 + c_{i\varepsilon}} < 1$ for $c_{i\varepsilon} \geq 0, i \in \{1, 2\}$, we can have
  \begin{align}
  \label{cient12}
    & - \beta \int_\Omega \tau_\varepsilon c_{1\varepsilon} \cdot \ln c_{1\varepsilon} \leq \frac{\beta}{e} \int_\Omega \tau_\varepsilon \leq \frac{\beta\tau_\ast}{e}|\Omega|,  \\[5pt]
  \label{cient13}
   & -  \beta \int_\Omega c_{2\varepsilon} c_{1\varepsilon} \cdot \ln c_{1\varepsilon} \leq \frac{\beta}{e} \int_\Omega c_{2\varepsilon},  \\[5pt]
  \label{cient14}
   & -  \int_\Omega \alpha_1(\chi_\varepsilon)  \tfrac{c_{1\varepsilon}}{1 + c_{1\varepsilon}} \cdot \ln c_{1\varepsilon}  \leq \frac{1}{e} \int_\Omega \alpha_1(\chi_\varepsilon) \tfrac{1}{1 + c_{1\varepsilon}} \leq \frac{1}{e}M_{\alpha_1} |\Omega|, \\[5pt]
  \label{cient15}
  & \int_\Omega  \alpha_2(\chi_\varepsilon) \tfrac{c_{2\varepsilon}}{1 + c_{2\varepsilon}}  \cdot \ln c_{1\varepsilon} \leq  M_{\alpha_2}\int_{\{c_{1\varepsilon} > 1\}}  \ln c_{1\varepsilon},\\[5pt]
  \label{cient16}
  & \int_\Omega   \alpha_2(\chi_\varepsilon) \tfrac{c_{2\varepsilon}}{1 + c_{2\varepsilon}} \leq  M_{\alpha_2}|\Omega|, \\[5pt]
  \label{cient17}
  & - \int_\Omega \alpha_1(\chi_\varepsilon) \tfrac{c_{1\varepsilon}}{1 + c_{1\varepsilon}} + \beta \int_\Omega c_{1\varepsilon}(1 - c_{1\varepsilon} -  c_{2\varepsilon} - \tau_\varepsilon) - \varepsilon \int_\Omega c_{1\varepsilon}^\theta\leq \beta  \int_\Omega c_{1\varepsilon}.
 \end{align}
Using \cite[Lemma 4.2]{stinner2014global} we can find two constants $C_1 > 0$ and $C_2 > 0$ such that
\begin{align}
  \label{stin1} & \beta \int_\Omega \{c_{1\varepsilon} \ln c_{1\varepsilon} - c_{1\varepsilon}^2 \ln c_{1\varepsilon} \} \leq -\frac{\beta}{2} \int_\Omega c_{1\varepsilon}^2 \ln(2 + c_{1\varepsilon}) + C_1  \\[5pt]
  \label{stin2} & - \varepsilon \int_\Omega c_{1\varepsilon}^\theta \ln c_{1\varepsilon} \leq - \frac{\varepsilon}{2} \int_\Omega c_{1\varepsilon}^\theta \ln(2 + c_{1\varepsilon}) + C_2.
\end{align}

\noindent
Inserting \eqref{cient12}-\eqref{stin2} in \eqref{cient11} yields \eqref{c1entro_ineq1}.
\end{proof}

\begin{Lemma}
\label{elemmac2}
   Let $\zeta > 0$. There exists a constant $C > 0$ such that any solution of \eqref{sys1} satisfies
  \begin{align}
    \label{eyc1}
    \dfrac{d}{dt} \int_\Omega c_{2 \varepsilon} \ln c_{2 \varepsilon} + a_2 \int_\Omega \dfrac{|\nabla c_{2 \varepsilon}|^2}{c_{2\varepsilon}}  + \frac{\varepsilon}{2} \int_\Omega c_{2\varepsilon}^\theta \ln(2 + c_{2\varepsilon}) & \leq  \frac{\zeta}{2}\int_\Omega c_{2\varepsilon}^2 + \frac{b_\chi^2}{2\zeta} \int_\Omega |\Delta \chi_\varepsilon|^2  + M_{\alpha_1} \int_{\{c_{2\varepsilon} > 1\}}\ln c_{2\varepsilon}  +  C
  \end{align}
for all $\varepsilon \in (0, 1)$ and all $t > 0$.
\end{Lemma}

\begin{proof}
  We test the second equation of \eqref{sys1} with $\ln c_{2\varepsilon}$ and integrate by parts to have
   \begin{align}
    \label{eyc2}
    \nonumber & \dfrac{d}{dt} \int_\Omega c_{2 \varepsilon} \ln c_{2 \varepsilon} = - a_2\int_\Omega \dfrac{|\nabla c_{2 \varepsilon}|^2}{c_{2\varepsilon}} +  b_\chi \int_\Omega \nabla \chi_\varepsilon \cdot \nabla c_{2\varepsilon} + \int_\Omega \alpha_1(\chi_\varepsilon) \tfrac{c_{1\varepsilon}}{1 + c_{1\varepsilon}} \ln c_{2 \varepsilon} \\[5pt]
    & - \int_\Omega \alpha_2(\chi_\varepsilon)  \tfrac{c_{2\varepsilon}}{1 + c_{2\varepsilon}} \ln c_{2 \varepsilon}  - \varepsilon \int_\Omega c_{2\varepsilon}^\theta \ln c_{2 \varepsilon} + \int_\Omega \alpha_1(\chi_\varepsilon) \tfrac{c_{1\varepsilon}}{1 + c_{1\varepsilon}} - \int_\Omega \alpha_2(\chi_\varepsilon)  \tfrac{c_{2\varepsilon}}{1 + c_{2\varepsilon}} - \varepsilon \int_\Omega c_{2\varepsilon}^\theta
   \end{align}
  for all $t > 0$.\\[-1.5ex]
  
  \noindent
  Consider the term arising from the taxis interaction, $b_\chi \int_\Omega \nabla c_{2\varepsilon} \cdot \nabla \chi_\varepsilon$. Integrating it by parts once more and applying Young's inequality then yields for a $\zeta > 0$:
  \begin{align}
  \label{bchi2}
      b_\chi \int_\Omega \nabla c_{2\varepsilon} \cdot \nabla \chi_\varepsilon = -b_\chi \int_\Omega c_{2\varepsilon} \Delta \chi_\varepsilon \leq \frac{\zeta}{2}\int_\Omega c_{2\varepsilon}^2 + \frac{b_\chi^2}{2\zeta} \int_\Omega |\Delta \chi_\varepsilon|^2 \quad \text{for all}~ t > 0.
  \end{align}
 Inserting \eqref{bchi2} into \eqref{eyc2} and handling the remaining terms as in Lemma \ref{entropy1}, we obtain \eqref{eyc1}.
\end{proof}

\begin{Lemma}
    \label{chi_lemma1}
    For all $t \in (0, T)$ the following holds true:
    \begin{align}
        \label{chi_ent}
    \frac{1}{2}\frac{d}{dt}\int_\Omega |\nabla \chi_\varepsilon|^2 + \frac{1}{2} \int_\Omega |\nabla \chi_\varepsilon|^2 + \frac{D_\chi}{2} \int_\Omega |\Delta \chi_\varepsilon|^2  \leq \frac{2a_\chi^2\chi_\infty^2}{D_\chi}\int_\Omega c_{1\varepsilon}^2 + \frac{2a_\chi^2\chi_\infty^2}{D_\chi}\int_\Omega c_{2\varepsilon}^2 + C^\ast
    \end{align}
    where $C^\ast := \frac{2}{D_\chi}\left(M_\chi^2 + \frac{\chi_\infty^2}{2}\right)|\Omega|$.
\end{Lemma}
\begin{proof}
    We test the third equation in \eqref{sys1} with $- \Delta \chi_\varepsilon$, which results in
    \begin{align*}
         \frac{1}{2}\frac{d}{dt}\int_\Omega |\nabla \chi_\varepsilon|^2 + D_\chi \int_\Omega |\Delta \chi_\varepsilon|^2  & = a_\chi \int_\Omega (c_{1\varepsilon} + c_{2\varepsilon}) \chi_\varepsilon \cdot \Delta \chi_\varepsilon - \int_\Omega F_\varepsilon(\chi_\varepsilon) \Delta \chi_\varepsilon \qquad \text{for all}~ t > 0.
    \end{align*}
    Adding $\tfrac{1}{2}\int_\Omega |\nabla \chi_\varepsilon|^2$ to both sides and using \eqref{mol_bound_short}, \eqref{chi_bound2}, together with Young's inequality in the standard manner, we obtain
    \begin{align*}
    & \frac{1}{2}\frac{d}{dt}\int_\Omega |\nabla \chi_\varepsilon|^2 + \frac{1}{2}\int_\Omega |\nabla \chi_\varepsilon|^2 + D_\chi \int_\Omega |\Delta \chi_\varepsilon|^2 = a_\chi \int_\Omega (c_{1\varepsilon} + c_{2\varepsilon}) \chi_\varepsilon \Delta \chi_\varepsilon -\int_\Omega F_\varepsilon(\chi_\varepsilon) \Delta \chi_\varepsilon + \frac{1}{2} \int_\Omega \nabla \chi_\varepsilon \cdot \nabla \chi_\varepsilon \\[5pt]
    & \leq a_\chi \chi_\infty \int_\Omega c_{1\varepsilon} \Delta \chi_\varepsilon + a_\chi \chi_\infty \int_\Omega c_{2\varepsilon} \Delta \chi_\varepsilon -  \int_\Omega F_\varepsilon(\chi_\varepsilon)\Delta \chi_\varepsilon - \frac{1}{2}\int_\Omega \chi_\varepsilon \Delta \chi_\varepsilon \\[5pt]
    & \leq \frac{D_\chi}{2} \int_\Omega |\Delta \chi_\varepsilon|^2 + \frac{2a_\chi^2\chi_\infty^2}{D_\chi}\int_\Omega c_{1\varepsilon}^2 + \frac{2a_\chi^2\chi_\infty^2}{D_\chi}\int_\Omega c_{2\varepsilon}^2 + \frac{2}{D_\chi}\left(M_\chi^2 + \frac{\chi_\infty^2}{4}\right)|\Omega|
    \end{align*}
 for all $t \in (0, T)$.
\end{proof}

\begin{Lemma}
  \label{tau_lemma}
 There exists a constant $C > 0$ such that any solution of \eqref{sys1} satisfies
  \begin{align}
   \label{tau_ent}
    \frac{1}{2} \frac{d}{dt} \int_\Omega \frac{|\nabla \tau_\varepsilon|^2}{\tau_\varepsilon} +  \frac{\mu}{2} \int_\Omega \frac{|\nabla \tau_\varepsilon|^2}{\tau_\varepsilon} +\frac{\varepsilon}{2} \int_\Omega \tau_\varepsilon |D^2 \ln \tau_\varepsilon|^2 +  \delta \int_\Omega \nabla \tau_\varepsilon \cdot \nabla c_{1\varepsilon} + \frac{\delta}{2} \int_\Omega c_{1 \varepsilon}\frac{|\nabla \tau_\varepsilon|^2}{\tau_\varepsilon} \leq \int_\Omega \frac{|\nabla c_{2\varepsilon}|^2}{c_{2\varepsilon}} + C
  \end{align}
 for all $\varepsilon \in (0, 1)$ and all $t > 0$.
\end{Lemma}

\begin{proof}
  Using the fourth equation in \eqref{sys1} and the positivity of the solution components, we can calculate
  \begin{align}
  \label{chi_tau1}
    \nonumber & \frac{1}{2} \frac{d}{dt} \int_\Omega \frac{|\nabla \tau_\varepsilon|^2}{\tau_\varepsilon} = \int_\Omega \frac{\nabla \tau_\varepsilon \cdot \nabla \tau_{\varepsilon t}}{\tau_\varepsilon} - \frac{1}{2} \int_\Omega \frac{|\nabla \tau_\varepsilon|^2}{\tau_\varepsilon^2} \tau_{\varepsilon t}   \\[5pt]
    \nonumber & = \int_\Omega \frac{\nabla \tau_\varepsilon }{\tau_\varepsilon} \cdot \nabla (\varepsilon \Delta \tau_\varepsilon - \delta c_{1\varepsilon}  \tau_\varepsilon - \mu \tau_\varepsilon +  \tfrac{c_{2\varepsilon}}{1 + c_{2\varepsilon}})
   - \frac{1}{2} \int_\Omega \frac{|\nabla \tau_\varepsilon|^2}{\tau_\varepsilon^2} (\varepsilon \Delta \tau_\varepsilon - \delta c_{1\varepsilon}  \tau_\varepsilon - \mu \tau_\varepsilon +  \tfrac{c_{2\varepsilon}}{1 + c_{2\varepsilon}} )   \\[5pt]
   \nonumber & = \varepsilon \int_\Omega \dfrac{\nabla \tau_\varepsilon}{\tau_\varepsilon} \cdot \nabla \Delta \tau_\varepsilon - \dfrac{\varepsilon}{2} \int_\Omega \dfrac{|\nabla \tau_\varepsilon|^2}{\tau_\varepsilon^2} \Delta \tau_\varepsilon - \delta \int_\Omega \nabla \tau_\varepsilon \cdot \nabla c_{1\varepsilon} - \frac{\mu}{2} \int_\Omega \frac{|\nabla \tau_\varepsilon|^2}{\tau_\varepsilon} - \frac{\delta}{2} \int_\Omega c_{1 \varepsilon} \frac{|\nabla \tau_\varepsilon|^2}{\tau_\varepsilon} \\[5pt]
   & + \int_\Omega \frac{\nabla \tau_\varepsilon \cdot \nabla c_{2\varepsilon}}{\tau_\varepsilon (1 + c_{2\varepsilon})^2} - \frac{1}{2} \int_\Omega \frac{|\nabla \tau_\varepsilon|^2}{\tau_\varepsilon^2} \frac{c_{2\varepsilon}}{1 + c_{2\varepsilon}}\quad \text{for all}~ t > 0.
  \end{align}
  We can handle the first two terms on the right-hand side of \eqref{chi_tau1} in the same manner as in \cite[Lemma 4.2, (4.9)-(4.12)]{dai2022global}
  \begin{equation}
    \label{chi_tau2}
    \varepsilon \int_{\Omega} \dfrac{\nabla \tau_\varepsilon}{\tau_\varepsilon} \cdot \nabla \Delta \tau_\varepsilon - \dfrac{\varepsilon}{2} \int_\Omega \dfrac{|\nabla \tau_\varepsilon|^2}{\tau_\varepsilon^2} \Delta \tau_\varepsilon \leq - \frac{\varepsilon}{2} \int_\Omega \tau_\varepsilon |D^2 \ln \tau_\varepsilon|^2 + C_1 .
  \end{equation}
   To handle the pair $- \frac{1}{2} \int_\Omega \frac{|\nabla \tau_\varepsilon|^2}{\tau_\varepsilon^2} \frac{c_{2\varepsilon}}{1 + c_{2\varepsilon}} + \int_\Omega \frac{\nabla \tau_\varepsilon \cdot \nabla c_{2\varepsilon}}{\tau_\varepsilon (1 + c_{2\varepsilon})^2} $ we will rely on the approach developed in \cite{surulescu2021does} (cf. the assumption in (1.9) and the result derived in (3.13) therein). It is evident that $\frac{1}{(1 + c_{2\varepsilon})^3} < 1$ for $c_{2\varepsilon} \geq 0$, hence
  \begin{align}
  \label{eh3}
   \nonumber & - \dfrac{1}{2} \int_\Omega \dfrac{|\nabla \tau_\varepsilon|^2}{\tau_\varepsilon^2} \dfrac{c_{2\varepsilon}}{1 + c_{2\varepsilon}} + \int_\Omega \dfrac{\nabla \tau_\varepsilon \cdot \nabla c_{2\varepsilon}}{\tau_\varepsilon (1 + c_{2\varepsilon})^2}\\[5pt]
    & \nonumber \leq - \dfrac{1}{2} \int_\Omega \dfrac{|\nabla \tau_\varepsilon|^2}{\tau_\varepsilon^2} \dfrac{c_{2\varepsilon}}{1 + c_{2\varepsilon}} + \dfrac{1}{2} \int_\Omega \dfrac{|\nabla \tau_\varepsilon|^2}{\tau_\varepsilon^2} \dfrac{c_{2\varepsilon}}{1 + c_{2\varepsilon}} + \dfrac{1}{2} \int_\Omega\dfrac{1}{(1 + c_{2\varepsilon})^4} \cdot \dfrac{1 + c_{2\varepsilon}}{c_{2\varepsilon}} |\nabla c_{2\varepsilon}|^2                                              \\[5pt]
  &= \dfrac{1}{2} \int_\Omega\dfrac{1}{(1 + c_{2\varepsilon})^3} \cdot \dfrac{|\nabla c_{2\varepsilon}|^2}{c_{2\varepsilon}} \leq  \int_\Omega \dfrac{|\nabla c_{2\varepsilon}|^2}{c_{2\varepsilon}}
  \end{align}
for all $t > 0$. Inserting \eqref{chi_tau2}-\eqref{eh3} into \eqref{chi_tau1} yields \eqref{tau_ent}.
\end{proof}

\noindent
We are now in a position to develop an entropy-type functional for the model under study.
\begin{Lemma}
     \label{entropy_lemma}
      Let $T > 0$. Then there exists a $C(T) > 0$ such that for all $\varepsilon \in (0, 1)$ any solution to \eqref{sys1} satisfies
  \begin{align}
    \label{entropy_main1}
    \frac{a_2 \delta}{4 b_\tau} \int_\Omega \left(c_{1\varepsilon} \ln c_{1\varepsilon} + \tfrac{1}{e}\right) + \int_\Omega \left(c_{2 \varepsilon} \ln c_{2 \varepsilon} + \tfrac{1}{e}\right) + \frac{b_\chi^2}{D_\chi \zeta} \int_\Omega |\nabla \chi_\varepsilon|^2 + \frac{a_2}{8} \int_\Omega \frac{|\nabla \tau_\varepsilon|^2}{\tau_\varepsilon} \leq C(T)
     \end{align}
     for all $t \in (0, T)$ and
    \begin{align}
    \label{entropy_main2}
    \nonumber & \frac{a_1 a_2 \delta}{8 b_\tau} \int_0^T \int_\Omega \frac{|\nabla c_{1\varepsilon}|^2}{c_{1\varepsilon}} + \frac{a_2}{8}  \int_0^T \int_\Omega \dfrac{|\nabla c_{2 \varepsilon}|^2}{c_{2\varepsilon}}+ \frac{a_2 \delta \beta}{8 b_\tau} \int_0^T \int_\Omega c_{1\varepsilon}^2 \ln(2 + c_{1\varepsilon}) + \frac{a_2 \delta \varepsilon}{8 b_\tau} \int_0^T \int_\Omega c_{1\varepsilon}^\theta \ln(2 + c_{1\varepsilon}) \\[5pt]
    &  + \frac{b_\chi^2}{2 \zeta} \int_0^T \int_\Omega |\Delta \chi_\varepsilon|^2  + \frac{a_2 \delta}{8} \int_0^T \int_\Omega c_{1\varepsilon} \frac{|\nabla \tau_\varepsilon|^2}{\tau_\varepsilon}   + \frac{\varepsilon}{2} \int_0^T \int_\Omega c_{2\varepsilon}^\theta \ln(2 + c_{2\varepsilon})  +\frac{a_2 \varepsilon}{8} \int_0^T \int_\Omega \tau_\varepsilon |D^2 \ln \tau_\varepsilon|^2 \leq C(T).
  \end{align}
\end{Lemma}
\begin{proof}
 We combine \eqref{l1c1} and \eqref{c1entro_ineq1} this results in
\begin{align}
\label{new_change1}
   & \nonumber  \frac{d}{dt} \int_\Omega \left(c_{1\varepsilon} \ln c_{1\varepsilon} + \tfrac{1}{e}\right) + \xi_1  \int_\Omega \left(c_{1\varepsilon} \ln c_{1\varepsilon} + \tfrac{1}{e}\right) + a_1 \int_\Omega \frac{|\nabla c_{1\varepsilon}|^2}{c_{1\varepsilon}} + \frac{\beta}{2} \int_\Omega c_{1\varepsilon}^2 \ln(2 + c_{1\varepsilon}) \\[5pt]
  & + \frac{\varepsilon}{2} \int_\Omega c_{1\varepsilon}^\theta \ln(2 + c_{1\varepsilon})  \leq b_\tau \int_\Omega  \nabla \tau_\varepsilon \cdot \nabla c_{1\varepsilon} +  M_{\alpha_2}\int_{\{ c_{1\varepsilon} > 1\}} \ln c_{1\varepsilon} + \xi_1  \int_\Omega c_{1\varepsilon} \ln c_{1\varepsilon} + \xi_1 \frac{|\Omega|}{e} + C_1
\end{align}
for any $\xi_1 > 0$ and for all $t \in (0, T)$. To estimate the term $\xi_1  \int_\Omega c_{1\varepsilon} \ln c_{1\varepsilon}$ on the right-hand side of \eqref{new_change1} we use the fact that $\xi_1  \int_\Omega c_{1\varepsilon} \ln c_{1\varepsilon} \leq \xi_1 \int_\Omega c_{1\varepsilon}^{\frac{3}{2}}$ for $c_{1\varepsilon} \geq 0$. As $n \leq 3$ we can use the Gagliardo-Nirenberg inequality in conjunction with \eqref{l1c11} as in \cite[(3.19) with $q = \frac{3}{2}$]{tao2021global} to have
\begin{align}
  \label{new_change2}
  \nonumber \xi_1 \int_\Omega c_{1\varepsilon}^{\frac{3}{2}} & \leq C_2 \xi_1 \|\nabla \sqrt{c_{1\varepsilon}}\|^{\frac{n}{2}}_{L^2(\Omega)} \|\sqrt{c_{1\varepsilon}}\|^{3 - \frac{n}{2}}_{L^2(\Omega)} + C_2 \xi_1 \|\nabla \sqrt{c_{1\varepsilon}}\|^3_{L^2(\Omega)} \\[5pt]
  \leq & \frac{C_3 \xi_1}{4} \int_\Omega \frac{|\nabla c_{1\varepsilon}|^2}{c_{1\varepsilon}} + C_4 \quad \text{for all}~ t \in (0, T),
\end{align}
where $C_3 > 0$ and $C_4 > 0$. Setting $\xi_1 = \frac{2 a_1}{C_3}$ in \eqref{new_change2} we have the following estimate
\begin{align}
  \label{new_change3}
  \xi_1 \int_\Omega c_{1\varepsilon}^{\frac{3}{2}} \leq \frac{a_1}{2} \int_\Omega \frac{|\nabla c_{1\varepsilon}|^2}{c_{1\varepsilon}} + C_4 \quad \text{for all}~ t \in (0, T).
\end{align}
Also using \eqref{l1c1} we can have
\begin{align}
  \label{new_change4}
  M_{\alpha_2}\int_{\{ c_{1\varepsilon} > 1\}} \ln c_{1\varepsilon} \leq M_{\alpha_2} \int_\Omega c_{1\varepsilon} \leq C_5  \quad \text{for all}~ t \in (0, T).
\end{align}
Inserting \eqref{new_change3} and \eqref{new_change4} in \eqref{new_change1} results in
\begin{align}
\label{new_change5}
   & \nonumber \frac{d}{dt} \int_\Omega \left(c_{1\varepsilon} \ln c_{1\varepsilon} + \tfrac{1}{e}\right) + \xi_1  \int_\Omega \left(c_{1\varepsilon} \ln c_{1\varepsilon} + \tfrac{1}{e}\right) + \frac{a_1}{2} \int_\Omega \frac{|\nabla c_{1\varepsilon}|^2}{c_{1\varepsilon}} + \frac{\beta}{2} \int_\Omega c_{1\varepsilon}^2 \ln(2 + c_{1\varepsilon}) \\[5pt]
   & + \frac{\varepsilon}{2} \int_\Omega c_{1\varepsilon}^\theta \ln(2 + c_{1\varepsilon})  \leq b_\tau \int_\Omega  \nabla \tau_\varepsilon \cdot \nabla c_{1\varepsilon}  + C_6 \qquad \text{for all}~ t \in (0, T).
\end{align}
Following the same procedure as above, from \eqref{l1c2} and \eqref{eyc1}, we obtain for any $\xi_2 > 0$ 
\begin{align}
\label{new_change6}
  \nonumber & \dfrac{d}{dt} \int_\Omega \left(c_{2 \varepsilon} \ln c_{2 \varepsilon} + \tfrac{1}{e}\right) + \xi_2 \int_\Omega \left(c_{2 \varepsilon} \ln c_{2 \varepsilon} + \tfrac{1}{e}\right) + \frac{a_2}{2} \int_\Omega \dfrac{|\nabla c_{2 \varepsilon}|^2}{c_{2\varepsilon}}  + \frac{\varepsilon}{2} \int_\Omega c_{2\varepsilon}^\theta \ln(2 + c_{2\varepsilon}) \\[5pt]
  & \leq  \frac{\zeta}{2}\int_\Omega c_{2\varepsilon}^2 + \frac{b_\chi^2}{2\zeta} \int_\Omega |\Delta \chi_\varepsilon|^2 + C_7 \qquad \text{for all}~ t \in (0, T)
\end{align}
and for any $\zeta > 0$. Multiplying \eqref{chi_ent} by $\frac{2 b_\chi^2}{D_\chi \zeta}$, \eqref{tau_ent} by $\frac{a_2}{4}$, and \eqref{new_change5} by $\frac{a_2 \delta}{4 b_\tau}$, and then adding them to \eqref{new_change6}, results in
\begin{align}
    \label{ebnt2}
    \nonumber & \frac{d}{dt}\left\{\frac{a_2 \delta}{4 b_\tau} \int_\Omega \left(c_{1\varepsilon} \ln c_{1\varepsilon} + \tfrac{1}{e}\right) + \int_\Omega \left(c_{2 \varepsilon} \ln c_{2 \varepsilon} + \tfrac{1}{e}\right) + \frac{b_\chi^2}{D_\chi \zeta} \int_\Omega |\nabla \chi_\varepsilon|^2 + \frac{a_2}{8} \int_\Omega \frac{|\nabla \tau_\varepsilon|^2}{\tau_\varepsilon}\right\} \\[5pt]
    \nonumber & + \frac{\xi_1 a_2 \delta}{4 b_\tau} \int_\Omega \left(c_{1\varepsilon} \ln c_{1\varepsilon} + \tfrac{1}{e}\right) + \xi_2 \int_\Omega \left(c_{2 \varepsilon} \ln c_{2 \varepsilon} + \tfrac{1}{e}\right) + \frac{b_\chi^2}{D_\chi \zeta} \int_\Omega |\nabla \chi_\varepsilon|^2 + \frac{\mu a_2}{8} \int_\Omega \frac{|\nabla \tau_\varepsilon|^2}{\tau_\varepsilon}+ \frac{a_1 a_2 \delta}{8 b_\tau} \int_\Omega \frac{|\nabla c_{1\varepsilon}|^2}{c_{1\varepsilon}}\\[5pt]
    \nonumber &  + \frac{a_2}{4}  \int_\Omega \dfrac{|\nabla c_{2 \varepsilon}|^2}{c_{2\varepsilon}} + \frac{b_\chi^2}{2 \zeta} \int_\Omega |\Delta \chi_\varepsilon|^2 + \frac{a_2 \delta}{8} \int_\Omega c_{1\varepsilon} \frac{|\nabla \tau_\varepsilon|^2}{\tau_\varepsilon} + \frac{a_2 \delta \beta}{8 b_\tau} \int_\Omega c_{1\varepsilon}^2 \ln(2 + c_{1\varepsilon}) + \frac{a_2 \delta \varepsilon}{8 b_\tau} \int_\Omega c_{1\varepsilon}^\theta \ln(2 + c_{1\varepsilon})  \\[5pt]
    & + \frac{\varepsilon}{2} \int_\Omega c_{2\varepsilon}^\theta \ln(2 + c_{2\varepsilon}) +\frac{a_2 \varepsilon}{8} \int_\Omega \tau_\varepsilon |D^2 \ln \tau_\varepsilon|^2  \leq \frac{4b_\chi^2 a_\chi^2 \chi_\infty^2}{D_\chi^2 \zeta} \int_\Omega c_{1\varepsilon}^2 + \left(\tfrac{\zeta}{2} +  \tfrac{4b_\chi^2 a_\chi^2 \chi_\infty^2}{D_\chi^2 \zeta}\right) \int_\Omega c_{2\varepsilon}^2
\end{align}
for any $\zeta > 0$ and for all $t \in (0, T)$. We apply the Gagliardo-Nirenberg inequality  \cite[(3.19) with $q = 2$]{tao2021global} for $n \in \{2, 3\}$ together with \eqref{l1c2} to directly obtain
 \begin{align}
 \label{c2_sq1}
    \nonumber & \int_\Omega c_{2\varepsilon}^2 \leq C_{GN}\cdot \left\|\nabla \sqrt{c_{2\varepsilon}}\right\|^n_{L^2(\Omega)}\left\|\sqrt{c_{2\varepsilon}}\right\|^{4 -
    n}_{L^2(\Omega)} + C_{GN} \|\sqrt{c_{2\varepsilon}}\|^4_{L^2(\Omega)}\\[5pt]
    & \int_\Omega c_{2\varepsilon}^2 \leq C_8(T) \int_\Omega \dfrac{|\nabla c_{2\varepsilon}|^2}{c_{2\varepsilon}} + C_9(T)  \quad \text{for all} ~  t \in (0, T).
  \end{align}
From above, we can have
\begin{align}
    \label{c234}
    \left(\tfrac{\zeta}{2} +  \tfrac{4b_\chi^2 a_\chi^2 \chi_\infty^2}{D_\chi^2 \zeta}\right) \int_\Omega c_{2\varepsilon}^2 \leq  \left(\tfrac{\zeta}{2} +  \tfrac{4b_\chi^2 a_\chi^2 \chi_\infty^2}{D_\chi^2 \zeta}\right)C_8(T) \int_\Omega \dfrac{|\nabla c_{2\varepsilon}|^2}{c_{2\varepsilon}} + C_{10}(T)  \quad \text{for all} ~  t \in (0, T).
\end{align}
Choose $\zeta$ such that
\begin{align}
    \label{c235}
    \left(\tfrac{\zeta}{2} +  \tfrac{4b_\chi^2 a_\chi^2 \chi_\infty^2}{D_\chi^2 \zeta}\right)C_8(T) := \frac{a_2}{8}.
\end{align}
We define two functions
\begin{align}
    \label{entro23}
    \mathcal{E}_\varepsilon(t) := \frac{a_2 \delta}{4 b_\tau} \int_\Omega \left(c_{1\varepsilon} \ln c_{1\varepsilon} + \tfrac{1}{e}\right) + \int_\Omega \left(c_{2 \varepsilon} \ln c_{2 \varepsilon} + \tfrac{1}{e}\right) + \frac{b_\chi^2}{D_\chi \zeta} \int_\Omega |\nabla \chi_\varepsilon|^2 + \frac{a_2}{8} \int_\Omega \frac{|\nabla \tau_\varepsilon|^2}{\tau_\varepsilon}
\end{align}
and
\begin{align}
    \label{entro45}
    \nonumber \mathcal{D}_\varepsilon(t) := & \frac{a_1 a_2 \delta}{8 b_\tau} \int_\Omega \frac{|\nabla c_{1\varepsilon}|^2}{c_{1\varepsilon}} + \frac{a_2}{8}  \int_\Omega \dfrac{|\nabla c_{2 \varepsilon}|^2}{c_{2\varepsilon}} + \frac{b_\chi^2}{2 \zeta} \int_\Omega |\Delta \chi_\varepsilon|^2 + \frac{a_2 \delta}{8} \int_\Omega c_{1\varepsilon} \frac{|\nabla \tau_\varepsilon|^2}{\tau_\varepsilon} + \frac{a_2 \delta \beta}{8 b_\tau} \int_\Omega c_{1\varepsilon}^2 \ln(2 + c_{1\varepsilon}) \\[5pt]
    & + \frac{a_2 \delta \varepsilon}{8 b_\tau} \int_\Omega c_{1\varepsilon}^\theta \ln(2 + c_{1\varepsilon}) + \frac{\varepsilon}{2} \int_\Omega c_{2\varepsilon}^\theta \ln(2 + c_{2\varepsilon}) +\frac{a_2 \varepsilon}{8} \int_\Omega \tau_\varepsilon |D^2 \ln \tau_\varepsilon|^2.
\end{align}
In view of \eqref{c235}, we can rewrite \eqref{ebnt2} by using \eqref{entro23} and \eqref{entro45}, with $\varrho \leq \min\{\xi_1, \xi_2, 1, \mu\}$, in the following form:
\begin{align}
\label{entropy_in3}
\mathcal{E}'_\varepsilon(t) + \varrho \mathcal{E}_\varepsilon(t) + \mathcal{D}_\varepsilon(t)
\leq \frac{4b_\chi^2 a_\chi^2 \chi_\infty^2}{D_\chi^2 \zeta} \int_\Omega c_{1\varepsilon}^2 + C_{11},
\quad \text{for all } t \in (0, T).
\end{align}
Clearly, $\mathcal{D}_\varepsilon(t)$ is positive, hence
\begin{align}
\label{ento_in51}
\mathcal{E}'_\varepsilon(t) + \varrho \mathcal{E}_\varepsilon(t)
\leq \frac{4b_\chi^2 a_\chi^2 \chi_\infty^2}{D_\chi^2 \zeta} \int_\Omega c_{1\varepsilon}^2 + C_{11},
\quad \text{for all } t \in (0, T).
\end{align}
In view of \eqref{c1_bound1}, we can find a constant $C_{12}(T) > 0$ such that
\begin{align*}
    \frac{4b_\chi^2 a_\chi^2 \chi_\infty^2}{D_\chi^2 \zeta} \int_0^T \int_\Omega c_{1\varepsilon}^2 \leq C_{12}(T).
\end{align*}
This allows us to apply \cite[Lemma 3.4]{stinner2014global} to \eqref{ento_in51} and obtain a $C_{13}(T) > 0$ such that
\begin{align}
\label{entropy_in5}
\mathcal{E}_\varepsilon(t) \leq C_{13}, \quad \text{for all } t \in (0, T),
\end{align}
which gives \eqref{entropy_main1}. A straightforward integration of \eqref{entropy_in3} over $(0, T)$ yields in view of \eqref{entropy_in5}
\begin{equation*}
\int_0^T \mathcal{D}_\varepsilon(t) \, dt \leq C_{14}(T),
\end{equation*}
which gives \eqref{entropy_main2}.
\end{proof}
\subsection{Further $\varepsilon$-independent estimates}

Based on inequalities \eqref{entropy_main1} and \eqref{entropy_main2} in Lemma \ref{entropy_lemma}, we now establish estimates for the solution components that will lead to strong compactness properties.

\begin{Lemma}
 \label{nab_bound_lem}
   Let $T > 0$. Then there exists a $C(T) > 0$ such that for any $\varepsilon \in (0, 1)$:
   \begin{align}
     \label{chi_tau_nab}  \int_\Omega |\nabla \chi_\varepsilon(\cdot, t)|^2 \leq & C(T),  \quad  \int_\Omega |\nabla \tau_\varepsilon(\cdot, t)|^2 \leq C(T),  \quad \text{for all} ~ t \in (0, T), \\[5pt]
     \label{c2_sq} & \int_0^T \int_\Omega c_{2\varepsilon}^2 \leq C(T).
   \end{align}
 \end{Lemma}
\begin{proof}
  The first estimate in \eqref{chi_tau_nab} is a direct consequence of \eqref{entropy_main1}. To validate the second estimate in \eqref{chi_tau_nab}, we apply \eqref{t_bound} and \eqref{entropy_main1}, which together ensure the existence of a constant $C(T) > 0$ such that
  \begin{align*}
    \int_\Omega |\nabla \tau_\varepsilon|^2 =  \int_\Omega \frac{|\nabla \tau_\varepsilon|^2}{\tau_\varepsilon} \cdot \tau_\varepsilon \leq \tau_\ast \int_\Omega \frac{|\nabla \tau_\varepsilon|^2}{\tau_\varepsilon} \leq C(T)  \quad \text{for all} ~ t \in (0, T).
  \end{align*}
 To prove \eqref{c2_sq}, we can use \eqref{c2_sq1} to have
 \begin{align}
 \label{c2_sq14}
   \int_\Omega c_{2\varepsilon}^2 \leq C_3 \int_\Omega \dfrac{|\nabla c_{2\varepsilon}|^2}{c_{2\varepsilon}} + C_4  \quad \text{for all} ~  t \in (0, T).
  \end{align}
 An integration of \eqref{c2_sq14} in view of \eqref{entropy_main2} results in
  \begin{align*}
    \int_0^T \int_\Omega c_{2\varepsilon}^2 \leq C_3 \int_0^T \int_\Omega \dfrac{|\nabla c_{2\varepsilon}|^2}{c_{2\varepsilon}} + C_4 T \leq C(T).
  \end{align*}
\end{proof}
\noindent
We will now derive appropriate compactness properties for the solution components of \eqref{sys1}, which will then imply convergence to a global weak solution of the original problem.

\begin{Lemma}
  \label{eind_ht}
  Let $T > 0$ and assume that $k > \frac{n + 2}{2}$ is an integer, then there exists a $C(T) > 0$ such that for all $\varepsilon \in (0, 1)$
\begin{align}
    & \label{chi_com_bound} \|\partial_t \chi_\varepsilon\|_{L^1((0, T); (W_0^{k, 2}(\Omega))^\ast)} \leq C(T),                \\[5pt]
    & \label{tau_com_bound} \|\partial_t \tau_\varepsilon\|_{L^1((0, T); (W_0^{k, 2}(\Omega))^\ast)} \leq C(T).
  \end{align}
\end{Lemma}
\begin{proof}
 To derive \eqref{chi_com_bound}, by using  \eqref{l1c1}, \eqref{l1c2}, \eqref{chi_bound2} and \eqref{entropy_main1} we can find a $C_1(T) > 0$ such that
\begin{align}
  \label{chihpre2}
    \nonumber & \left|\int_\Omega \partial_t \chi_\varepsilon \cdot \psi\right|  = \left| D_\chi \int_\Omega \nabla \chi_\varepsilon \cdot \nabla \psi - a_\chi \int_\Omega c_{1\varepsilon} \chi_\varepsilon \psi - a_\chi \int_\Omega c_{2\varepsilon} \chi_\varepsilon \psi + \int_\Omega F_\varepsilon(\chi_\varepsilon) \psi \right|  \\[7pt]
    \nonumber & \leq \frac{D_\chi}{2} \left(\int_\Omega |\nabla \chi_\varepsilon|^2 + |\Omega|\right)  \|\nabla \psi\|_{L^\infty(\Omega)} + a_\chi \left(\int_\Omega  c_{1\varepsilon}\right) \chi_\infty  \|\psi\|_{L^\infty(\Omega)} + \left(\int_\Omega  c_{2\varepsilon} \right) \chi_\infty\|\psi\|_{L^\infty(\Omega)} + M_\chi |\Omega|   \|\psi\|_{L^\infty(\Omega)}  \\[5pt]
    & \leq C_1\|\psi\|_{W^{1, \infty}(\Omega)} ~~\text{for all}~ t \in (0, T) ~ \text{and each}~ \psi \in C^\infty_{0}(\Omega).
  \end{align}
 The Sobolev embedding theorem allows us to have
  \begin{equation}
   \label{hpre3}
    \|\psi\|_{W^{1, \infty}(\Omega)} \leq C_2\|\psi\|_{W^{k, 2}_0(\Omega)} ~~\text{for all}~ \psi \in C^\infty_{0}(\Omega)
  \end{equation}
  with $C_2 >0$. From \eqref{chihpre2} and \eqref{hpre3} we can directly get
  \begin{equation*}
   \|\partial_t \chi_\varepsilon\|_{L^1(0, T; (W_0^{k, 2}(\Omega))^\ast)} = \int_0^T \sup_{\substack{\psi \in C_0^\infty(\Omega),\\ \|\psi\|_{W^{k, 2}_0(\Omega)} = 1}} \left|\int_\Omega \partial_t \chi_\varepsilon \cdot \psi\right| dt \leq C_1 C_2 T,
  \end{equation*}
 thus establishing \eqref{chi_com_bound}. To derive \eqref{tau_com_bound}, by using  \eqref{l1c1}, \eqref{t_bound}, \eqref{entropy_main1} and the fact that $0 \leq \tfrac{c_{2\varepsilon}}{1 + c_{2\varepsilon}} < 1$ for $c_{2\varepsilon} \geq 0$, we can find a $C_3(T) > 0$ such that
   \begin{align*}
    \nonumber \left|\int_\Omega \partial_t \tau_\varepsilon \cdot \psi\right| & = \left| \varepsilon \int_\Omega \nabla \tau_\varepsilon \cdot \nabla \psi - \delta \int_\Omega \tau_\varepsilon c_{1\varepsilon}\psi  - \mu \int_\Omega \tau_\varepsilon \psi +  \int_\Omega \dfrac{c_{2\varepsilon}}{1 + c_{2\varepsilon}} \psi\right|  \\[5pt]
    & \nonumber \leq \varepsilon \left(\int_\Omega \dfrac{|\nabla \tau_\varepsilon|^2}{\tau_\varepsilon}\right)^{\frac{1}{2}} \cdot \left(\int_\Omega \tau_\varepsilon\right)^{\frac{1}{2}} \|\nabla \psi\|_{L^\infty(\Omega)}
    + \delta \left(\int_\Omega \tau_\varepsilon c_{1\varepsilon}\right)\|\psi\|_{L^\infty(\Omega)}\\[7pt]
    & \nonumber  + \mu \left(\int_\Omega \tau_\varepsilon \right)\|\psi\|_{L^\infty(\Omega)} + \left(\int_\Omega\dfrac{c_{2\varepsilon}}{1 + c_{2\varepsilon}}\right) \|\psi\|_{L^\infty(\Omega)} \\[7pt]
    & \nonumber \leq \varepsilon \dfrac{\sqrt{\tau_\ast |\Omega|}}{2} \left(\int_\Omega \dfrac{|\nabla \tau_\varepsilon|^2}{\tau_\varepsilon} + 1\right)  \|\nabla \psi\|_{L^\infty(\Omega)} + \delta \tau_\ast \left(\int_\Omega c_{1\varepsilon}\right)\|\psi\|_{L^\infty(\Omega)} \\[7pt]
    & \nonumber + \mu \tau_\ast |\Omega| \|\psi\|_{L^\infty(\Omega)} + |\Omega| \|\psi\|_{L^\infty(\Omega)} \\[7pt]
    & \leq C_3\|\psi\|_{W^{1, \infty}(\Omega)} ~~\text{for all}~ t \in (0, T) ~ \text{and each}~ \psi \in C^\infty_{0}(\Omega).
  \end{align*}
  We can now apply a similar reasoning as in the derivation of \eqref{chi_com_bound} to obtain \eqref{tau_com_bound}.
\end{proof}
\begin{Lemma}
  \label{c1_tim_der}
  Let $T > 0$ and $l > \frac{n + 4}{2}$ be an integer. Then there exists a $C(T) > 0$ such that for all $\varepsilon \in (0, 1)$
  \begin{align}
    \label{c1_tim1} & \|c_{1\varepsilon}\|_{L^{\frac{4}{3}}(0, T; W^{1, \frac{4}{3}}(\Omega))} \leq C(T),   \\[5pt]
    \label{c1_tim2} & \|\partial_t c_{1\varepsilon}\|_{L^{1}(0, T; (W^{l, 2}_0(\Omega))^\ast)} \leq C(T), \quad \text{and}, \\[5pt]
    \label{c2_tim1} & \|c_{2\varepsilon}\|_{L^{\frac{5}{4}}(0, T; W^{1, \frac{5}{4}}(\Omega))} \leq C(T),   \\[5pt]
    \label{c2_tim2} & \|\partial_t c_{2\varepsilon}\|_{L^{1}(0, T; (W^{l, 2}_0(\Omega))^\ast)} \leq C(T).
  \end{align}
\end{Lemma}

\begin{proof}
  From \eqref{c1_bound1} and \eqref{entropy_main2} we infer that there exists a $C_1(T) > 0$ with
  \begin{align*}
    \int_0^{T} \int_\Omega c_{1\varepsilon}^2 \leq C_1(T) \quad ~\text{and}~ \quad \int_0^T \int_\Omega \dfrac{|\nabla c_{1 \varepsilon}|^2}{c_{1\varepsilon}} \leq C_1(T).
  \end{align*}
  This directly yields
  \begin{equation}
    \label{c1_tres1}
    \int_0^T \int_\Omega |\nabla c_{1\varepsilon}|^{\frac{4}{3}} \leq \left(\int_0^T \int_\Omega \dfrac{|\nabla c_{1 \varepsilon}|^2}{c_{1\varepsilon}}\right)^{\frac{2}{3}} \cdot  \left( \int_0^{T} \int_\Omega c_{1\varepsilon}^2 \right)^{\frac{1}{3}} \leq C_2(T).
  \end{equation}
  The H\"{o}lder inequality then ensures the boundedness of $\|c_{1\varepsilon}\|_{L^{\frac{4}{3}}(\Omega \times (0, T))}$, thus establishing \eqref{c1_tim1}.\\[-2ex]
  
  \noindent
  To prove \eqref{c1_tim2}, using \eqref{alpha_def}, \eqref{t_bound}, the fact that $\frac{c_{i\varepsilon}}{1 + c_{i\varepsilon}} < 1$ for $c_{i\varepsilon} \geq 0,\ i \in \{1, 2\}$ and Young's inequality we can have for all $\psi \in C_0^\infty(\Omega)$
  \begin{align}
  \label{c1_tim5}
    \nonumber \left|\int_\Omega \partial_t c_{1\varepsilon} \cdot \psi\right| & = \bigg| - a_1 \int_\Omega \nabla c_{1\varepsilon} \cdot \nabla \psi + b_\tau \int_\Omega c_{1\varepsilon} \nabla \tau_\varepsilon \cdot \nabla \psi  - \int_\Omega \alpha_1(\chi_\varepsilon)\tfrac{c_{1\varepsilon}}{1 + c_{1\varepsilon}}  \psi + \int_\Omega \alpha_2( \chi_\varepsilon) \tfrac{c_{2\varepsilon}}{1 + c_{2\varepsilon}} \psi \\[5pt]
    \nonumber &   + \beta \int_\Omega  c_{1\varepsilon}\psi -  \beta \int_\Omega  c_{1\varepsilon}^2\psi - \beta \int_\Omega  c_{1\varepsilon}   c_{2\varepsilon}\psi  - \beta \int_\Omega  c_{1\varepsilon} \tau_\varepsilon\psi  - \varepsilon \int_\Omega  c_{1\varepsilon}^\theta \psi \bigg|\\[5pt]
    \nonumber & \leq a_1 \left(\frac{3}{4} \int_\Omega |\nabla c_{1\varepsilon}|^{\frac{4}{3}} + \frac{1}{4} |\Omega|\right)\|\nabla \psi\|_{L^\infty(\Omega)} + \frac{b_\tau}{2} \left(\int_\Omega c_{1\varepsilon}^2 + \int_\Omega |\nabla \tau_\varepsilon|^2\right)\|\nabla \psi\|_{L^\infty(\Omega)} \\[5pt]
    \nonumber & + M_{\alpha_1} |\Omega| \psi\|_{L^\infty(\Omega)} +  M_{\alpha_2} |\Omega| \| \psi\|_{L^\infty(\Omega)}  + \beta \left(\int_\Omega c_{1\varepsilon}\right)\| \psi\|_{L^\infty(\Omega)}+ \beta \left(\int_\Omega c_{1\varepsilon}^2\right)\| \psi\|_{L^\infty(\Omega)} \\[5pt]
    \nonumber &   + \frac{\beta}{2} \left(\int_\Omega c_{1\varepsilon}^2 + \int_\Omega c_{2\varepsilon}^2\right)\| \psi\|_{L^\infty(\Omega)} + \beta \tau_\ast \left(\int_\Omega c_{1\varepsilon}\right)\| \psi\|_{L^\infty(\Omega)} + \varepsilon \left(\int_\Omega c_{1\varepsilon}^\theta\right)\| \psi\|_{L^\infty(\Omega)} \\[5pt]
    & \leq Z_\varepsilon(t) \|\psi\|_{W^{2, \infty}(\Omega)} \quad \text{for all} ~ t \in (0, T),
  \end{align}
  where
  \begin{align*}
    Z_\varepsilon(t) := & \frac{3}{4} a_1 \int_\Omega |\nabla c_{1\varepsilon}|^{\frac{4}{3}} + \frac{b_\tau}{2} \int_\Omega |\nabla \tau_\varepsilon|^{2} + \left(\frac{b_\tau}{2} + \frac{3\beta}{2}\right) \int_\Omega c_{1\varepsilon}^2 + \frac{\beta}{2} \int_\Omega c_{2\varepsilon}^2  + \beta(1 + \tau_\ast) \int_\Omega c_{1\varepsilon} \\[5pt]
    & + \varepsilon \int_\Omega c_{1\varepsilon}^\theta + (M_{\alpha_1} + M_{\alpha_2} + \tfrac{1}{4}) |\Omega|
  \end{align*}
for all $\varepsilon \in (0, 1)$ and each $t \in (0, T)$. Using \eqref{l1c1}, \eqref{c1_bound1}, \eqref{c1_bound2}, \eqref{chi_tau_nab}, \eqref{c2_sq} and \eqref{c1_tres1}, we can find a $C_3(T) > 0$ such that
\begin{equation}
 \label{c1_tim6}
  \int_0^T Z_\varepsilon(t) dt \leq C_3(T).
\end{equation}
\noindent
From \eqref{c1_tim5}, \eqref{c1_tim6} and the Sobolev embedding $W_0^{l, 2}(\Omega) \hookrightarrow W^{2, \infty}(\Omega)$ we can have
\begin{align*}
   \|\partial_t c_{1\varepsilon}\|_{L^1(0, T; (W_0^{l, 2}(\Omega))^\ast)} & = \int_0^T \sup_{\substack{\psi \in C_0^\infty(\Omega),\\ \|\psi\|_{W^{l, 2}_0(\Omega)} = 1}} \left|\int_\Omega \partial_t c_{1\varepsilon} \cdot \psi \right| dt  \leq \int_0^T \sup_{\substack{\psi \in C_0^\infty(\Omega),\\ \|\psi\|_{W^{l, 2}_0(\Omega)} = 1}} Z_\varepsilon(t) \|\psi\|_{W^{2, \infty}(\Omega)} dt  \\[5pt]
   & \leq C_3(T) C_4
  \end{align*}
which entails \eqref{c1_tim2}.\\[-2ex]

\noindent
As $n \leq 3$, by the Gagliardo-Nirenberg inequality and \eqref{l1c2} we can have
\begin{align}
\label{c2_tim6}
  \nonumber\int_\Omega c_{2\varepsilon}^{\frac{5}{3}} & \leq C_5 \|\nabla \sqrt{ c_{2\varepsilon}}\|_{L^2(\Omega)}^{\frac{2n}{3}} \|\sqrt{c_{2\varepsilon}}\|^{\frac{10 - 2n}{3}}_{L^2(\Omega)} + C_6 \|\sqrt{c_{2\varepsilon}}\|^{\frac{10}{3}}_{L^2(\Omega)}                   \\[5pt]
  \nonumber  & \leq C_7 \left(\int_\Omega |\nabla \sqrt{ c_{2\varepsilon}}|^2 \right)^{\frac{n}{3}} + C_7   \\[5pt]
  & \leq C_8 \int_\Omega |\nabla \sqrt{ c_{2\varepsilon}}|^2 + C_8 \leq C_9 \int_\Omega \frac{|\nabla c_{2\varepsilon}|^2}{c_{2\varepsilon}} + C_9 \quad \text{for all} ~ t \in (0, T).
\end{align}
Integrating \eqref{c2_tim6} over $(0,T)$, in view of \eqref{entropy_main2} we can have
\begin{align}
  \label{c2_tim8}
  \int_0^T \int_\Omega c_{2\varepsilon}^{\frac{5}{3}}  \leq C_9 \int_0^T \int_\Omega \frac{|\nabla c_{2\varepsilon}|^2}{c_{2\varepsilon}} + C_9(T) \leq C_{10}(T).
\end{align}
\eqref{c2_tim8} taken together with \eqref{entropy_main2} allow us to get
\begin{equation}
\label{c2_tim9}
  \int_0^T \int_\Omega |\nabla c_{2\varepsilon}|^{\frac{5}{4}} \leq \left(\int_0^T \int_\Omega \frac{|\nabla c_{2\varepsilon}|^2}{c_{2\varepsilon}}\right)^{\frac{5}{8}} \cdot \left(\int_0^T \int_\Omega c_{2\varepsilon}^{\frac{5}{3}}\right)^{\frac{3}{8}} \leq C_{11}(T).
\end{equation}
By H\"{o}lder's inequality we can then deduce \eqref{c2_tim1} from \eqref{c2_tim8} and \eqref{c2_tim9}.\\[-2ex]

\noindent
Now for proving the estimate \eqref{c2_tim2}, use \eqref{alpha_def}, the fact that $\frac{c_{i\varepsilon}}{1 + c_{i\varepsilon}} < 1$ for $c_{i\varepsilon} \geq 0,\ i \in \{1, 2\}$ and Young's inequality, to have for all $\psi \in C_0^\infty(\Omega)$
  \begin{align}
  \label{c2_tim5}
    \nonumber \left|\int_\Omega \partial_t c_{2\varepsilon} \cdot \psi\right| & = \bigg| - a_2 \int_\Omega \nabla c_{2\varepsilon} \cdot \nabla \psi + b_\chi \int_\Omega c_{2\varepsilon} \nabla \chi_\varepsilon \cdot \nabla \psi  + \int_\Omega \alpha_1(\chi_\varepsilon)\tfrac{c_{1\varepsilon}}{1 + c_{1\varepsilon}}  \psi - \int_\Omega \alpha_2(\chi_\varepsilon)\tfrac{c_{1\varepsilon}}{1 + c_{1\varepsilon}} \psi  \bigg| \\[5pt]
    \nonumber & \leq a_2 \left(\frac{4}{5} \int_\Omega |\nabla c_{2\varepsilon}|^{\frac{5}{4}} + \frac{1}{5} |\Omega|\right)\|\nabla \psi\|_{L^\infty(\Omega)} + \frac{b_\chi}{2} \left(\int_\Omega c_{2\varepsilon}^2 + \int_\Omega |\nabla \chi_\varepsilon|^2\right)\|\nabla \psi\|_{L^\infty(\Omega)} \\[5pt]
    \nonumber & + M_{\alpha_1} |\Omega| \| \psi\|_{L^\infty(\Omega)} +  M_{\alpha_2}  |\Omega| \| \psi\|_{L^\infty(\Omega)}\\[5pt]
    & \leq V_\varepsilon(t) \|\psi\|_{W^{2, \infty}(\Omega)} \quad \text{for all} ~ t \in (0, T),
  \end{align}
   where
  \begin{align*}
    V_\varepsilon(t) := & \frac{4}{5} a_2 \int_\Omega |\nabla c_{2\varepsilon}|^{\frac{5}{4}} + \frac{b_\chi}{2} \int_\Omega |\nabla \chi_\varepsilon|^2 + \frac{b_\chi}{2} \int_\Omega c_{2\varepsilon}^2 + (M_{\alpha_1} + M_{\alpha_2} + \tfrac{1}{5})|\Omega|
  \end{align*}
for all $\varepsilon \in (0, 1)$ and each $t \in (0, T)$. Using \eqref{chi_tau_nab}, \eqref{c2_sq} and \eqref{c2_tim9}, we can find a $C_{12}(T) > 0$ such that
\begin{equation}
 \label{c2_tim611}
  \int_0^T V_\varepsilon(t) dt \leq C_{12}(T).
\end{equation}
From \eqref{c2_tim5}, \eqref{c2_tim611}, and the Sobolev embedding $W_0^{l, 2}(\Omega) \hookrightarrow W^{2, \infty}(\Omega)$ we can have
\begin{align*}
   \|\partial_t c_{2\varepsilon}\|_{L^1(0, T; (W_0^{l, 2}(\Omega))^\ast)} & = \int_0^T \sup_{\substack{\psi \in C_0^\infty(\Omega),\\ \|\psi\|_{W^{l, 2}_0(\Omega)} = 1}} \left|\int_\Omega \partial_t c_{2\varepsilon} \cdot \psi \right| dt  \leq \int_0^T \sup_{\substack{\psi \in C_0^\infty(\Omega),\\ \|\psi\|_{W^{l, 2}_0(\Omega)} = 1}} V_\varepsilon(t) \|\psi\|_{W^{2, \infty}(\Omega)} dt  \\[5pt]
   & \leq C_{12}(T) C_{13},
  \end{align*}
which entails \eqref{c2_tim2}.
\end{proof}

\noindent
Using Lemma \ref{eind_ht} and Lemma \ref{c1_tim_der} in conjunction with the  Aubin-Lions lemma \cite[Chapter 3]{temam2001navier}, we can establish the following strong precompactness properties.

\begin{Lemma}
\label{precompactness_all}
Let $T > 0$. Then
\begin{align}
  \label{precompact_c1}    & \{c_{1\varepsilon}\}_{\varepsilon \in (0,1)} ~\text{is strongly precompact in}~ L^{\frac{4}{3}}(\Omega \times (0, T)),  \\[5pt]
  \label{precompact_c2}    & \{c_{2\varepsilon}\}_{\varepsilon \in (0,1)} ~\text{is strongly precompact in}~ L^{\frac{5}{4}}(\Omega \times (0, T)),  \\[5pt]
  \label{precompact_chi}     & \{\chi_{\varepsilon}\}_{\varepsilon \in (0,1)}   ~\text{is strongly precompact in}~ L^{2}(\Omega \times (0, T))   \quad ~\text{and}       \\[5pt]
  \label{precompact_tau}   & \{\tau_{\varepsilon}\}_{\varepsilon \in (0,1)}   ~\text{is strongly precompact in}~ L^{2}(\Omega \times (0, T)).
\end{align}
\end{Lemma}

\begin{proof}
  Let $l$ be the integer chosen in \eqref{c1_tim2}. Note that $(W_0^{l, 2}(\Omega))^\ast$ is a Hilbert space, and the embedding $W^{1, \frac{4}{3}}(\Omega) \hookrightarrow L^{\frac{4}{3}}(\Omega)$ is compact. Hence, by the Aubin-Lions lemma \cite[Theorem 3.2.3 and Remark 3.2.1]{temam2001navier}, together with \eqref{c1_tim1} and \eqref{c1_tim2}, we conclude that $\{c_{1\varepsilon}\}_{\varepsilon \in (0,1)}$ is strongly precompact in $L^{\frac{4}{3}}(\Omega \times (0, T))$, thus establishing \eqref{precompact_c1}. A similar argument applies to \eqref{precompact_c2}, \eqref{precompact_chi}, and \eqref{precompact_tau}.
\end{proof}


\section{Construction of weak solutions}\label{sec:limits}

\begin{proof}[Proof of Theorem \ref{main_theorem}]
By \eqref{entropy_main2}, we conclude that
\begin{align*}
\{c_{1\varepsilon}^2 \ln (2 + c_{1\varepsilon})\}_{\varepsilon \in (0, 1)} \quad \text{and} \quad
\{\varepsilon c_{1\varepsilon}^\theta \ln (2 + c_{1\varepsilon})\}_{\varepsilon \in (0, 1)}
\end{align*}
are bounded in $L^1_{loc}(\bar{\Omega} \times [0, \infty))$, which implies that
\begin{align*}
  \{c_{1\varepsilon}^2\}_{\varepsilon \in (0, 1)} \quad \text{and} \quad
\{\varepsilon c_{1\varepsilon}^\theta\}_{\varepsilon \in (0, 1)}
\end{align*}
are equi-integrable. By the Dunford-Pettis theorem \cite[A8.14]{alt2016linear}, these sequences are weakly sequentially precompact in $L^1_{\mathrm{loc}}(\bar{\Omega} \times [0, \infty))$. Moreover, from \eqref{entropy_main2}, the sequence $\{c_{1\varepsilon}\}_{\varepsilon \in (0, 1)}$ is bounded in $L^2_{loc}(\bar{\Omega} \times [0, \infty))$. The above reasoning, together with \eqref{c1_tim1} and \eqref{precompact_c1}, allows us to apply a standard extraction procedure to select a subsequence $\{\varepsilon_j\}_{j \in \mathbb{N}}$ such that
\begin{equation*}
  \varepsilon_j \in (0, 1) ~\text{for all}~ j \in \mathbb{N}, ~\text{and}~ \varepsilon_j \searrow 0 ~\text{as}~ j \to \infty
\end{equation*}
and a nonnegative function
\begin{equation*}
\label{th3c1_main}
c_1 \in L^2_{loc}(\bar{\Omega} \times [0, \infty)) \cap L_{loc}^{\frac{4}{3}}([0, \infty) ; W^{1, \frac{4}{3}}(\Omega)),
\end{equation*}
\begin{equation}
\begin{cases}
 \label{mth1}
  c_{1\varepsilon_j} & \to c_1 ~\text{a.e. in}~ \Omega \times (0, \infty),                                                                          \\[5pt]
  c_{1\varepsilon_j}^2 & \rightharpoonup c_1^2 ~\text{in}~ L^1_{loc}(\bar{\Omega} \times [0, \infty)),                                              \\[5pt]
  c_{1\varepsilon_j} & \rightharpoonup c_1 ~\text{in}~ L^2_{loc}(\bar{\Omega} \times [0, \infty)),                                                  \\[5pt]
  \nabla c_{1\varepsilon_j} & \rightharpoonup \nabla c_1 ~\text{in}~ L^{\frac{4}{3}}_{loc}(\bar{\Omega} \times [0, \infty)),                        \\[5pt]
  \varepsilon_j c_{1\varepsilon_j}^\theta & \to 0 ~\text{a.e. in}~ \Omega \times (0, \infty),      \quad \text{and}                                    \\[5pt]
  \varepsilon_j c_{1\varepsilon_j}^\theta & \rightharpoonup 0 ~\text{in}~ L_{loc}^1(\bar{\Omega} \times [0, \infty))
\end{cases}
\end{equation}
as $j \to \infty$. With $1$ as a test function, from \eqref{mth1} we can directly have
\begin{equation}
  \label{mth411}
   c_{1\varepsilon_j} \to c_1 ~\text{in}~ L^2_{loc}(\bar{\Omega} \times [0, \infty)) ~ \text{as} ~ j \to \infty.
\end{equation}
Similarly, using \eqref{entropy_main1}, \eqref{entropy_main2}, \eqref{chi_tau_nab}, \eqref{c2_tim1}, and \eqref{precompact_c2}--\eqref{precompact_tau}, and passing to a subsequence if necessary, we obtain
\begin{align}
  \label{mth4} & c_{2\varepsilon_j} \to c_2 ~\text{in}~ L^{\frac{5}{4}}_{loc}(\bar{\Omega} \times [0, \infty)) ~\text{and a.e. in}~ \Omega \times (0, \infty),                        \\[5pt]
  \label{mth5} & \nabla c_{2\varepsilon_j}  \rightharpoonup \nabla c_2 ~\text{in}~ L^{\frac{5}{4}}_{loc}(\bar{\Omega} \times [0, \infty)),                                            \\[5pt]
  \label{eps2} & \varepsilon_j c_{2\varepsilon_j}^\theta \rightharpoonup 0 ~\text{in}~ L^1_{loc}(\bar{\Omega} \times [0, \infty)),                                                    \\[5pt]
  \label{mth6} & \chi_{\varepsilon_j} \to \chi ~\text{in}~ L^{2}_{loc}(\bar{\Omega} \times [0, \infty)) ~\text{and a.e. in}~ \Omega \times (0, \infty),                                     \\[5pt]
  \label{mth7} & \nabla \chi_{\varepsilon_j}  \rightharpoonup \nabla \chi ~\text{in}~ L^{2}_{loc}(\bar{\Omega} \times [0, \infty)),                                                         \\[5pt]
  \label{mth8} & \tau_{\varepsilon_j} \to \tau ~\text{in}~ L^{2}_{loc}(\bar{\Omega} \times [0, \infty)) ~\text{and a.e. in}~ \Omega \times (0, \infty),  \quad ~\text{and}            \\[5pt]
  \label{mth9} & \nabla \tau_{\varepsilon_j}  \rightharpoonup \nabla \tau ~\text{in}~ L^{2}_{loc}(\bar{\Omega} \times [0, \infty))
\end{align}
as $j \to \infty$, here $c_2, \chi$ and $\tau$ are nonnegative functions such that
\begin{equation}
\label{2mth1}
  \begin{cases}
    c_2 \in L^{\frac{5}{4}}_{loc}([0, \infty); W^{1, \frac{5}{4}}(\Omega)),                          \\[5pt]
    \chi \in  L^2_{loc}([0, \infty) ; W^{1, 2}(\Omega)),        \\[5pt]
    \tau \in L^\infty(\Omega \times (0, \infty)) \cap L^2_{loc}([0, \infty) ; W^{1, 2}(\Omega)).
  \end{cases}
\end{equation}
Let $T > 0$ and $\psi \in C^\infty_0(\bar{\Omega} \times [0, T))$ with $\dfrac{\partial \psi}{\partial \nu} = 0$ on $\partial \Omega \times [0, T)$, then from the first equation of \eqref{sys1} we have for all $\varepsilon \in (0, 1)$

\begin{align}
   \label{mthdefeq1}
  \nonumber  & - \int_0^T \int_\Omega c_{1\varepsilon} \partial_t \psi - \int_\Omega c_{10\varepsilon} \psi(\cdot, 0) = - a_1 \int_0^T \int_\Omega \nabla c_{1\varepsilon} \cdot \nabla \psi + b_\tau \int_0^T \int_\Omega c_{1\varepsilon} \nabla \tau_\varepsilon \cdot \nabla \psi - \int_0^T \int_\Omega \alpha_1(\chi_\varepsilon )\tfrac{c_{1\varepsilon}}{1 + c_{1\varepsilon}} \psi \\[5pt]
  & + \int_0^T \int_\Omega \alpha_2(\chi_\varepsilon )\tfrac{c_{2\varepsilon}}{1 + c_{2\varepsilon}} \psi + \beta \int_0^T \int_\Omega c_{1\varepsilon} \psi - \beta \int_0^T \int_\Omega  c_{1\varepsilon}^2 \psi - \beta \int_0^T \int_\Omega  c_{1\varepsilon} c_{2\varepsilon} \psi - \beta \int_0^T \int_\Omega \tau_\varepsilon c_{1\varepsilon}\psi - \varepsilon \int_0^T \int_\Omega c_{1\varepsilon}^\theta \psi.
\end{align}
Combining \eqref{reg_assumptions}, \eqref{mth1} and \eqref{mth411} we can deduce that
\begin{equation}
  \label{weak_sol1}
  - \int_0^T \int_\Omega c_{1\varepsilon} \partial_t \psi - \int_\Omega c_{10\varepsilon} \psi(\cdot, 0) \to - \int_0^T \int_\Omega c_{1} \partial_t \psi - \int_\Omega c_{10} \psi(\cdot, 0),
\end{equation}
and
\begin{equation}
\label{weak_sol2}
 - a_1 \int_0^T \int_\Omega \nabla c_{1\varepsilon} \cdot \nabla \psi \to  - a_1 \int_0^T \int_\Omega \nabla c_{1} \cdot \nabla \psi,
\end{equation}
as well as
\begin{align}
\label{weak_sol3}
 - \varepsilon \int_0^T \int_\Omega c_{1\varepsilon}^\theta \psi \to 0
\end{align}
as $\varepsilon = \varepsilon_j \searrow 0$.
From \eqref{mth1}, \eqref{mth411} and \eqref{mth9} we can claim that
\begin{align}
\label{weak_sol4}
 b_\tau \int_0^T \int_\Omega c_{1\varepsilon} \nabla \tau_\varepsilon \cdot \nabla \psi \to b_1 \int_0^T \int_\Omega c_{1} \nabla \tau \cdot \nabla \psi
\end{align}
as $\varepsilon = \varepsilon_j \searrow 0$. Taken together, \eqref{alpha_def}, \eqref{mth411}, \eqref{mth4}, and the uniform boundedness of functions of the form $\frac{c_{i\varepsilon}}{1 + c_{i\varepsilon}} < 1$, $i \in \{1, 2\}$, yield
\begin{align}
 \label{weak_sol6}
 - \int_0^T \int_\Omega \alpha_1(\chi_\varepsilon )\frac{c_{1\varepsilon}}{1 +  c_{1\varepsilon}} \psi \to - \int_0^T \int_\Omega \alpha_1(\chi)\frac{c_{1}}{1 +  c_{1}} \psi
\end{align}
and
\begin{align}
 \label{weak_sol7}
 \int_0^T \int_\Omega \alpha_2(\chi_\varepsilon )\frac{c_{2\varepsilon}}{1 +  c_{2\varepsilon}} \psi \to \int_0^T \int_\Omega \alpha_2(\chi)\frac{c_{2}}{1 +  c_{2}} \psi
\end{align}
as $\varepsilon = \varepsilon_j \searrow 0$. Taken together, \eqref{l1c1} and \eqref{c2_sq} result in
\begin{align*}
  c_{1\varepsilon} c_{2\varepsilon} \to c_{1} c_{2} \quad \text{in}~L^{1}(\Omega \times (0, T)),
\end{align*}
as $\varepsilon = \varepsilon_j \searrow 0$, which implies that
\begin{align}
 \label{weak_sol8}
  - \beta \int_0^T \int_\Omega c_{1\varepsilon}c_{2\varepsilon} \psi \to -\beta \int_0^T \int_\Omega c_{1} c_{2} \psi.
\end{align}
Combining \eqref{t_bound}, \eqref{mth1}, \eqref{mth411} and \cite[Lemma 3.9]{tao2021global} we can have
\begin{align}
 \label{weak_sol9}
  - \beta \int_0^T \int_\Omega c_{1\varepsilon} \tau_\varepsilon \psi \to -\beta \int_0^T \int_\Omega c_{1} \tau \psi.
\end{align}
Finally \eqref{mth1} and \eqref{mth411} yield
\begin{align}
  \label{cweak_sol1}
   \beta \int_0^T \int_\Omega c_{1\varepsilon} \psi - \beta \int_0^T \int_\Omega  c_{1\varepsilon}^2 \psi \to  \beta \int_0^T \int_\Omega c_{1} \psi - \beta \int_0^T \int_\Omega  c_{1}^2 \psi
\end{align}
as $\varepsilon = \varepsilon_j \searrow 0$.  The convergence results \eqref{weak_sol1}-\eqref{cweak_sol1} enable us to pass to the limit $\varepsilon = \varepsilon_j \searrow 0$ in \eqref{mthdefeq1} to get
\begin{align}
   \label{wmthdefeq1}
  \nonumber  & - \int_0^T \int_\Omega c_{1} \partial_t \psi - \int_\Omega c_{10} \psi(\cdot, 0) = - a_1 \int_0^T \int_\Omega \nabla c_{1} \cdot \nabla \psi + b_\tau \int_0^T \int_\Omega c_{1} \nabla \tau \cdot \nabla \psi - \int_0^T \int_\Omega \alpha_1(\chi)\tfrac{c_{1}}{1 + c_{1}} \psi\\[5pt]
  & + \int_0^T \int_\Omega \alpha_2(\chi)\tfrac{c_{2}}{1 + c_2} \psi + \beta \int_0^T \int_\Omega c_{1}(1 - c_{1}  -  c_{2}  - \tau )\psi.
\end{align}
From the second equation of \eqref{sys1} we have for all $\varepsilon \in (0, 1)$
\begin{align}
   \label{mthdefeq2}
  \nonumber  & - \int_0^T \int_\Omega c_{2\varepsilon} \partial_t \psi - \int_\Omega c_{20\varepsilon} \psi(\cdot, 0) = - a_2 \int_0^T \int_\Omega \nabla c_{2\varepsilon} \cdot \nabla \psi  - b_\chi \int_0^T \int_\Omega c_{2\varepsilon} \chi_\varepsilon \Delta \psi \\[5pt]
  & - b_\chi \int_0^T \int_\Omega \chi_\varepsilon \nabla c_{2 \varepsilon} \cdot \nabla \psi + \int_0^T \int_\Omega \alpha_1(\chi_\varepsilon )\frac{c_{1\varepsilon}}{1 + c_{1\varepsilon}} \psi - \int_0^T \int_\Omega \alpha_2(\chi_\varepsilon )\frac{c_{2\varepsilon}}{1 + c_{2\varepsilon}} \psi - \varepsilon \int_0^T \int_\Omega c_{2\varepsilon}^\theta \psi.
\end{align}
Combining \eqref{reg_assumptions}, \eqref{mth4}, \eqref{mth5} and \eqref{eps2} we can deduce that
\begin{equation}
  \label{c2weak_sol1}
  - \int_0^T \int_\Omega c_{2\varepsilon} \partial_t \psi - \int_\Omega c_{20\varepsilon} \psi(\cdot, 0) \to - \int_0^T \int_\Omega c_{2} \partial_t \psi - \int_\Omega c_{20} \psi(\cdot, 0),
\end{equation}
and
\begin{equation}
\label{c2weak_sol2}
 - a_2 \int_0^T \int_\Omega \nabla c_{2\varepsilon} \cdot \nabla \psi \to  - a_2 \int_0^T \int_\Omega \nabla c_{2} \cdot \nabla \psi,
\end{equation}
as well as
\begin{align}
\label{c2weak_sol3}
 - \varepsilon \int_0^T \int_\Omega c_{2\varepsilon}^\theta \psi \to 0
\end{align}
as $\varepsilon = \varepsilon_j \searrow 0$. Collecting \eqref{chi_bound2}, \eqref{mth4}, \eqref{mth6} and \cite[Lemma 3.9]{tao2021global}
\begin{equation*}
   \chi_\varepsilon \partial_i \psi \to \chi \partial_i \psi \quad ~\text{in}~ L^{\frac{5}{4}}(\Omega \times (0, T)),
\end{equation*}
for all $i = 1, 2, \ldots n$, as $\varepsilon = \varepsilon_j \searrow 0$, this along with \eqref{mth5} implies that
\begin{equation}
\label{c2weak1}
  - b_\chi \int_0^T \int_\Omega \chi_\varepsilon \nabla c_{2 \varepsilon} \cdot \nabla \psi \to  - b_\chi \int_0^T \int_\Omega \chi \nabla c_{2} \cdot \nabla \psi.
\end{equation}
as $\varepsilon = \varepsilon_j \searrow 0$. From \eqref{mth4} and \eqref{mth6} we can have
\begin{equation}
\label{c2weak2}
   - b_\chi \int_0^T \int_\Omega c_{2\varepsilon} \chi_\varepsilon \Delta \psi \to  - b_\chi \int_0^T \int_\Omega c_{2} \chi \Delta \psi
\end{equation}
as $\varepsilon = \varepsilon_j \searrow 0$. The convergence of the remaining two terms can be established using the same reasoning as in the derivation of \eqref{weak_sol6} and \eqref{weak_sol7}.
The convergence results \eqref{c2weak_sol1}-\eqref{c2weak2} enable us to pass the limit $\varepsilon = \varepsilon_j \searrow 0$ in \eqref{mthdefeq2} to have
\begin{align}
   \label{weakmthdefeq2}
  \nonumber  & - \int_0^T \int_\Omega c_{2} \partial_t \psi - \int_\Omega c_{20} \psi(\cdot, 0) = - a_2 \int_0^T \int_\Omega \nabla c_{2} \cdot \nabla \psi  - b_\chi \int_0^T \int_\Omega c_{2} \chi \Delta \psi \\[5pt]
  & - b_\chi \int_0^T \int_\Omega \chi \nabla c_{2 } \cdot \nabla \psi + \int_0^T \int_\Omega \alpha_1(\chi)\tfrac{c_{1}}{1 + c_{1}} \psi - \int_0^T \int_\Omega \alpha_2(\chi)\tfrac{c_{2}}{1 + c_{2}} \psi.
\end{align}
From the third equation of \eqref{sys1} we have for all $\varepsilon \in (0, 1)$
\begin{align}
   \label{mthdefeq3}
  - \int_0^T \int_\Omega \chi_{\varepsilon} \partial_t \psi - \int_\Omega \chi_{0\varepsilon} \psi(\cdot, 0) = - D_\chi \int_0^T \int_\Omega \nabla \chi_{\varepsilon} \cdot \nabla \psi - a_\chi \int_0^T \int_\Omega c_{1\varepsilon}\chi_\varepsilon \psi - a_\chi \int_0^T \int_\Omega c_{2\varepsilon}\chi_\varepsilon \psi + \int_0^T \int_\Omega F_\varepsilon(\chi_\varepsilon).
\end{align}
Combining \eqref{reg_assumptions}, \eqref{mth6} and \eqref{mth7} we can deduce that
\begin{equation}
  \label{chiweak_sol1}
  - \int_0^T \int_\Omega \chi_{\varepsilon} \partial_t \psi - \int_\Omega \chi_{0\varepsilon} \psi(\cdot, 0) \to - \int_0^T \int_\Omega \chi \partial_t \psi - \int_\Omega \chi \psi(\cdot, 0),
\end{equation}
and
\begin{equation}
\label{chiweak_sol2}
 - D_\chi \int_0^T \int_\Omega \nabla \chi_{\varepsilon} \cdot \nabla \psi \to  - D_\chi \int_0^T \int_\Omega \nabla \chi \cdot \nabla \psi
\end{equation}
as $\varepsilon = \varepsilon_j \searrow 0$. Together, \eqref{mth411} and \eqref{mth6} allow us to have
\begin{equation}
  \label{chi_weak1}
  - a_\chi \int_0^T \int_\Omega c_{1\varepsilon}\chi_\varepsilon \psi  \to - a_\chi \int_0^T \int_\Omega c_{1}\chi \psi
\end{equation}
as $\varepsilon = \varepsilon_j \searrow 0$. Collecting \eqref{chi_bound2}, \eqref{mth4}, \eqref{mth6} and \cite[Lemma 3.9]{tao2021global} together yield
\begin{equation*}
  c_{2\varepsilon}\chi_\varepsilon  \to c_{2}\chi \quad ~\text{in}~ L^{\frac{5}{4}}(\Omega \times (0, T)),
\end{equation*}
as $\varepsilon = \varepsilon_j \searrow 0$, which implies
\begin{equation}
\label{chi_weak2}
  - a_\chi \int_0^T \int_\Omega c_{2\varepsilon}\chi_\varepsilon \psi \to - a_\chi \int_0^T \int_\Omega c_{2}\chi \psi.
\end{equation}
From \cite[Theorem 7, Appendix C: Calculus, particularly C5]{evans2022partial} and \eqref{mth7} we can have
\begin{align}
    \label{weak89}
     \int_0^T \int_\Omega F_\varepsilon(\chi_\varepsilon) \to \int_0^T \int_\Omega F(\chi).
\end{align}
The convergence results \eqref{chiweak_sol1}-\eqref{weak89} enable us to pass the limit $\varepsilon = \varepsilon_j \searrow 0$ in \eqref{mthdefeq3} to have
\begin{align}
   \label{chi_convgmthdefeq3}
  - \int_0^T \int_\Omega \chi \partial_t \psi - \int_\Omega \chi_{0} \psi(\cdot, 0) = - D_\chi \int_0^T \int_\Omega \nabla \chi \cdot \nabla \psi - a_\chi \int_0^T \int_\Omega c_{1}\chi \psi - a_\chi \int_0^T \int_\Omega c_{2}\chi \psi + \int_0^T \int_\Omega F(\chi).
\end{align}
From the fourth equation of \eqref{sys1} we have for all $\varepsilon \in (0, 1)$
\begin{align}
   \label{taumthdefeq3}
  - \int_0^T \int_\Omega \tau_{\varepsilon} \partial_t \psi - \int_\Omega \tau_{0\varepsilon} \psi(\cdot, 0) & = - \varepsilon \int_0^T \int_\Omega \nabla \tau_{\varepsilon} \cdot \nabla \psi
  - \delta \int_0^T \int_\Omega  \tau_\varepsilon c_{1\varepsilon} \psi - \mu \int_0^T \int_\Omega  \tau_\varepsilon \psi  +  \int_0^T \int_\Omega \dfrac{c_{2\varepsilon}}{1 + c_{2\varepsilon}} \psi.
\end{align}
We will first handle the artificial term on the right-hand side. It follows from \eqref{chi_bound2}, \eqref{entropy_main1} and the Cauchy-Schwarz inequality that
\begin{align}
\label{artificial1}
  \nonumber \left| - \varepsilon \int_0^T \int_\Omega \nabla \tau_\varepsilon \cdot \nabla \psi \right| & \leq  \varepsilon \left(\int_0^T \int_\Omega \frac{|\nabla \tau_\varepsilon|^2}{\tau_\varepsilon}\right)^{\frac{1}{2}}\left(\int_0^T \int_\Omega \tau_\varepsilon\right)^{\frac{1}{2}}\|\nabla \psi\|_{L^\infty(\Omega)} \\[5pt]
  \nonumber  & \leq \varepsilon T \sqrt{\tau_\ast |\Omega|}\left(\sup_{t \in (0, T)}  \int_\Omega\frac{|\nabla \tau_\varepsilon|^2}{\tau_\varepsilon}\right) \|\nabla \psi\|_{L^\infty(\Omega)}  \\[5pt]
  & \to 0 ~\text{as}~ \varepsilon = \varepsilon_j \searrow 0.
\end{align}
The convergence of the remaining terms on the right-hand side of \eqref{taumthdefeq3} can be established by collecting \eqref{reg_assumptions}, \eqref{mth411}, \eqref{mth4} and \eqref{mth8}. The convergence results  enable us to pass the limit $\varepsilon = \varepsilon_j \searrow 0$ in \eqref{taumthdefeq3} to have
\begin{align}
   \label{taumthdefeq4}
   - \int_0^T \int_\Omega \tau \partial_t \psi - \int_\Omega \tau_{0} \psi(\cdot, 0) = - \delta \int_0^T \int_\Omega \tau c_{1} \psi - \mu \int_0^T \int_\Omega \tau \psi - \int_0^T \int_\Omega \frac{c_{2}}{1 + c_{2}} \psi .
\end{align}
By collecting \eqref{wmthdefeq1}, \eqref{weakmthdefeq2}, \eqref{chi_convgmthdefeq3} and \eqref{taumthdefeq4}, we complete the proof.
\end{proof}

\section*{Acknowledgement}
Funded by the Deutsche Forschungsgemeinschaft (DFG, German Research Foundation) - grant number 465242756, within the SPP 2311-\textit{Robust coupling of continuum-biomechanical in silico models to establish active biological system models for later use in clinical applications - Co-design of modeling, numerics and usability}. NM gratefully acknowledges partial support by the TU Nachwuchsring at the RPTU.

\phantomsection
\printbibliography
\end{document}